\documentclass[10pt,oneside,english]{amsart}
\usepackage[LGR,T1]{fontenc}
\usepackage[latin9]{inputenc}
\usepackage[a4paper]{geometry}
\usepackage[active]{srcltx}
\usepackage{color}
\usepackage{textcomp}
\usepackage{amsthm}
\usepackage{amstext}
\usepackage{amssymb}
\usepackage{setspace}
\usepackage{esint}
\makeatletter

\DeclareRobustCommand{\greektext}{%
  \fontencoding{LGR}\selectfont\def\encodingdefault{LGR}}
\DeclareRobustCommand{\textgreek}[1]{\leavevmode{\greektext #1}}
\DeclareFontEncoding{LGR}{}{}
\DeclareTextSymbol{\~}{LGR}{126}

\numberwithin{equation}{section}
\numberwithin{figure}{section}
  \theoremstyle{plain}
  \newtheorem*{thm*}{\protect\theoremname}
\theoremstyle{plain}
\newtheorem{thm}{\protect\theoremname}[section]
  \theoremstyle{plain}
  \newtheorem{conjecture}[thm]{\protect\conjecturename}
  \theoremstyle{remark}
  \newtheorem*{rem*}{\protect\remarkname}
  \theoremstyle{plain}
  \newtheorem{cor}[thm]{\protect\corollaryname}
  \theoremstyle{remark}
  \newtheorem*{acknowledgement*}{\protect\acknowledgementname}
  \theoremstyle{plain}
  \newtheorem{prop}[thm]{\protect\propositionname}
 \theoremstyle{definition}
 \newtheorem*{defn*}{\protect\definitionname}
  \theoremstyle{plain}
  \newtheorem{lem}[thm]{\protect\lemmaname}
  \theoremstyle{definition}
  \newtheorem*{example*}{\protect\examplename}
  \theoremstyle{plain}
  \newtheorem*{fact*}{\protect\factname}

\makeatother

\usepackage{babel}
  \providecommand{\acknowledgementname}{Acknowledgement}
  \providecommand{\conjecturename}{Conjecture}
  \providecommand{\corollaryname}{Corollary}
  \providecommand{\definitionname}{Definition}
  \providecommand{\examplename}{Example}
  \providecommand{\factname}{Fact}
  \providecommand{\lemmaname}{Lemma}
  \providecommand{\propositionname}{Proposition}
  \providecommand{\remarkname}{Remark}
  \providecommand{\theoremname}{Theorem}
\providecommand{\theoremname}{Theorem}

\begin{document}

\title{Representations of reductive groups distinguished by symmetric subgroups}

\author{Itay Glazer}

\address{Faculty of Mathematics and Computer Science, Weizmann Institute of
Science, 234 Herzl Street, Rehovot 76100, Israel.}

\email{itay.glazer@weizmann.ac.il}
\begin{abstract}
\begin{singlespace}
\noindent Let $G$ be a complex connected reductive group, $G^{\theta}$
be its fixed point subgroup under a Galois involution $\theta$ and
$H$ be an open subgroup of $G^{\theta}$. We show that any $H$-distinguished
representation $\pi$ satisfies: 
\end{singlespace}

\noindent 1) $\pi^{\theta}\simeq\tilde{\pi}$, where $\tilde{\pi}$
is the contragredient representation and $\pi^{\theta}$ is the twist
of $\pi$ under $\theta$.

\noindent 2) $\mathrm{dim}_{\mathbb{C}}\left(\pi^{*}\right)^{H}\leq\left|B\backslash G/H\right|$,
where $B$ is a Borel subgroup of $G$.

By proving the first statement, we give a partial answer to a conjecture
by Prasad and Lapid. 
\end{abstract}

\maketitle
\begin{small}\tableofcontents{}

\end{small}

\section{\label{sec:Introduction}Introduction}

Let $\underline{G}$ be a connected reductive algebraic group defined
over $\mathbb{R}$, $\underline{G}^{\theta}$ be its fixed point subgroup
under some involution $\theta$ defined over $\mathbb{R}$, $G=\underline{G}(\mathbb{R})$
and $G^{\theta}=\underline{G}^{\theta}(\mathbb{R})$. We call $(G,H)$
a real symmetric pair if $H$ is an open subgroup of $G^{\theta}$.
Denote by $\mathcal{SAF}(G)$ the category of all finitely generated,
smooth, admissible, moderate growth, Fréchet representations of $G$
with continuous linear $G$-maps as morphisms. Denote by $\mathcal{SAF}_{\mathrm{Irr}}(G)$
the set of irreducible objects in $\mathcal{SAF}(G)$. A fundamental
task in the representation theory of the pair $(G,H)$ is to explore
representations $\pi\in\mathcal{SAF}_{\mathrm{Irr}}(G)$ embedded
in $C^{\infty}(G/H)$. By Frobenius reciprocity, the irreducible representations
embedded in $C^{\infty}(G/H)=\mathrm{Ind}_{H}^{G}(\mathbb{C})$ are
exactly the irreducible representations $\pi$ with a non trivial
$H$-invariant continuous functional, i.e $\pi$ with $\left(\pi^{*}\right)^{H}\neq0$.
Such representations are called $\emph{H-distinguished}$. Hence,
when studying the symmetric pair $(G,H)$, we would like to answer
the following questions: 
\begin{enumerate}
\item Which representations $\pi\in\mathcal{SAF}_{\mathrm{Irr}}(G)$ are
$H$-distinguished? 
\item What can we say about the multiplicity of such representation $\pi$,
i.e, the number $dim_{\mathbb{C}}\left(\pi^{*}\right)^{H}$?
\end{enumerate}
The first step towards answering the first question is to find a necessary
condition for a representation $\pi$ to be $H$-distinguished. The
following theorem is a version of a classical theorem of Gelfand and
Kazhdan (see \cite{GK75}) and it suggests a certain candidate for
such a condition.
\begin{thm*}
(Special case of Gelfand-Kazhdan criterion) Let $(G,H)$ be a real
symmetric pair, where $H=\underline{G}^{\theta}(\mathbb{R})$, and
let $\sigma$ be the anti-involution $\sigma(g)=\theta(g)^{-1}$.
Assume that $\sigma(\xi)=\xi$ for all bi $H$-invariant Schwartz
distributions $\xi$ on $G$. Then any $H$-distinguished representation
$\pi\in\mathcal{SAF}_{\mathrm{Irr}}(G)$ satisfies $\widetilde{\pi}\simeq\pi^{\theta}$,
where $\widetilde{\pi}$ denotes the contragredient representation
and $\pi^{\theta}$ is the twist of $\pi$ by $\theta$.
\end{thm*}
We see that the symmetric pairs $(G,H)$ that satisfy the above G-K
criterion have the property:
\[
\left(\pi^{*}\right)^{H}\neq0\,\Longrightarrow\,\widetilde{\pi}\simeq\pi^{\theta},
\]
for any $\pi\in\mathcal{SAF}_{\mathrm{Irr}}(G)$. As many symmetric
pairs satisfy the above G-K criterion, it is therefore natural to
ask for a symmetric pair $(G,H)$ if the condition $\widetilde{\pi}\simeq\pi^{\theta}$
is necessary and sufficient for $\pi$ to be $H$-distinguished. Using
(e.g. \cite[Theorem 8.2.1]{AG09}), one can deduce that this condition
is not sufficient. Indeed, any representation $\pi\in\mathcal{SAF}_{\mathrm{Irr}}(\mathrm{GL_{2n}}(\mathbb{C}))$
satisfies $\widetilde{\pi}\simeq\pi^{\theta}$, where $\theta(g):=\Omega^{-1}\left(g^{t}\right)^{-1}\Omega$,
$\Omega=\left(\begin{array}{cc}
0 & I_{n}\\
-I_{n} & 0
\end{array}\right)$. In particular, representations that are not $\mathrm{Sp_{2n}}(\mathbb{C})$-distinguished
satisfy $\widetilde{\pi}\simeq\pi^{\theta}$ (see \cite{GSS15} for
existence of such representations).

Although the condition $\widetilde{\pi}\simeq\pi^{\theta}$ is not
sufficient, it is conjectured that it is necessary when $G$ is a
complex reductive group. 
\begin{conjecture}
\label{Main conj}Let $(G,H)$ be a real symmetric pair, where $G$
is a complex connected reductive group, and let $\pi\in\mathcal{SAF}_{\mathrm{Irr}}(G)$
be $H$-distinguished. Then $\tilde{\pi}\simeq\pi^{\theta}$.
\end{conjecture}
This conjecture can be generalized to the following conjecture by
Lapid:
\begin{conjecture}
(Lapid) Let $\underline{G}$ be a connected algebraic group defined
over a local field $F$, $\theta:\underline{G}\longrightarrow\underline{G}$
an involution defined over $F$, $\underline{H}=\underline{G}^{\theta}$
and write $G=\underline{G}(F)$, $H=\underline{H}(F)$. Let $\pi$
be either in $\mathcal{SAF}_{\mathrm{Irr}}(G)$ in the Archimedean
case, or irreducible and smooth in the non-Archimedean case. Assume
that $\pi$ is $H$-distinguished, then the $L$-packet of $\pi$
is invariant with respect to the functor $\pi\longmapsto\widetilde{\pi}\circ\theta$.
\end{conjecture}
These conjectures have been proven for the pairs $(\mathrm{GL_{n}}(\mathbb{C}),\mathrm{U}(p,q))$
and their non-Archimedean analogue in \cite{ALOF12}, and for the
pair $(\mathrm{GL_{n}}(\mathbb{C}),\mathrm{GL_{n}}(\mathbb{R}))$
in \cite{Kem15}. 

In \cite{Pra}, Prasad formulates several conjectures about $\underline{G}(F)$-distinguished
representations of $\underline{G}(E)$ in terms of the Langlands parameters
of the representations of $\underline{G}(E)$, where $E/F$ is a quadratic
extension of local fields. He introduces a character $\omega_{G}:\underline{G}(F)\longrightarrow\mathbb{Z}/2\mathbb{Z}$
(see\cite[Section 8]{Pra}) that appears in many questions about distinction.
Hence, it is sometimes more interesting to consider representations
of $\underline{G}(E)$ that are $(\omega_{G},\underline{G}(F))$-distinguished.
Conjecture 3 of \cite{Pra} is part of the motivation to consider
Galois pairs  $(\underline{G}(\mathbb{C}),\underline{G}(\mathbb{R}))$
and in particular to discuss representations $\pi\in\mathcal{SAF}_{\mathrm{Irr}}(\underline{G}(\mathbb{C}))$
that are $(\omega,\underline{G}(\mathbb{R}))$-distinguished for any
character $\omega$ that is trivial on the connected component $\underline{G}(\mathbb{R})^{0}$
of $\underline{G}(\mathbb{R})$. In this work we present a partial
answer to Conjecture \ref{Main conj} and the first part of Conjecture
3 in \cite{Pra}.

A vast study has been done regarding the multiplicity of $H$-distinguished
representations, which is the second question we are dealing in this
paper. In \cite{Ban87} it has been shown that the multiplicity $dim_{\mathbb{C}}\left(\pi^{*}\right)^{H}$
is finite for real symmetric spaces. In \cite{KO13,KS16,AGM16}, several
bounds on the multiplicities have been established for real spherical
pairs. In the non-Archimedean case, finiteness of $dim_{\mathbb{C}}\left(\pi^{*}\right)^{H}$
has been shown for symmetric pairs in \cite{Del10} and certain spherical
pairs in \cite{SV}.

\subsection{\label{sub:Main-results-and}Main results}

We call $(G,H)$ a $Galois\,symmetric\,pair$ if $G$ is obtained
by restriction of scalars of the complexification of $G^{\theta}$
(e.g $G^{\theta}$ is a real form of $G$), and $H$ is an open subgroup
of $G^{\theta}$. In this paper we prove Conjecture \ref{Main conj}
and present a bound on $dim_{\mathbb{C}}\left(\pi^{*}\right)^{H}$,
for the case of Galois pairs $(G,H)$. We also deduce the first part
of Conjecture 3 of \cite{Pra}, namely that for any $\pi\in\mathcal{SAF}_{\mathrm{Irr}}(\underline{G}(\mathbb{C}))$
that is $(\omega_{G},\underline{G}(\mathbb{R}))$-distinguished it
holds that $\widetilde{\pi}\simeq\pi^{\theta}$, where $\tilde{\pi}$
is the contragredient representation, $\theta$ is the Galois involution
of $\underline{G}(\mathbb{C})$ fixing $\underline{G}(\mathbb{R})$
and $\pi^{\theta}$ is the twist of $\pi$ by $\theta$. 

Although we are mostly interested in Galois pairs, we work in a slightly
more general setting, where $G$ is a complex reductive group and
$\theta$ is any real involution; let $\underline{G}$ be a connected
reductive algebraic group defined over $\mathbb{R}$ and $\underline{G}_{\mathbb{C}/\mathbb{R}}$
be the restriction of scalars of the complexification of $\underline{G}$
(see Appendix \ref{Append:Restriction-of-scalars} for more details)
such that $G=\underline{G}(\mathbb{C})=\underline{G}_{\mathbb{C}/\mathbb{R}}(\mathbb{R})$.
Let $\theta$ be an $\mathbb{R}$-involution of $\underline{G}_{\mathbb{C}/\mathbb{R}}$
and write $G^{\theta}=\underline{G}_{\mathbb{C}/\mathbb{R}}^{\theta}(\mathbb{R})$. 
\begin{rem*}
A motivating example for the above construction is the pair $(\mathrm{GL_{n}}(\mathbb{C}),\mathrm{U}(n))$,
as the involution $\theta(g)=(g^{*})^{-1}$ is defined over $\mathbb{R}$
only after we apply restriction of scalars to $\mathrm{GL_{n}}$. 
\end{rem*}
We fix a Borel subgroup $B$, a $\theta$-stable maximal torus $T$,
and a $\theta$-stable maximal $\mathbb{R}$-split torus $A$ in $G$,
such that $A\subseteq T\subseteq B$ (see Lemma \ref{lemma 5.1}).
Let $N_{G}(T)$ be the normalizer of $T$ in $G$ and $\delta_{B}$
be the modular character of $B$. We further assume that $\theta$
is an involution that satisfies the following condition $(\star)$
\begin{equation}
\forall n\in N_{G}(T)\text{ such that }\theta(n)=n^{-1}\text{ we have that }\delta_{B^{\theta_{n}}}=\delta_{B}^{1/2}|_{B^{\theta_{n}}},\tag{\ensuremath{\star}}\label{eq:star}
\end{equation}
where $\theta_{n}$ is the involution of $G$ defined by $\theta_{n}(g)=n\theta(g)n^{-1}$
and $B^{\theta_{n}}$ is the fixed points subgroup of $B$ under $\theta_{n}$.
We prove the following theorem:
\begin{thm}
\label{thm 1.3} Let $G$ be a complex reductive group, $\theta$
be any involution of $G$ that satisfies property $(\star)$, $H$
be an open subgroup of $G^{\theta}$ and $\pi\in\mathcal{SAF}_{\mathrm{Irr}}(G)$
be $H$-distinguished. Then:

a) $\tilde{\pi}\simeq\pi^{\theta}$.

b) $\mathrm{dim}_{\mathbb{C}}\left(\pi^{*}\right)^{H}\leq\left|B\backslash G/H\right|$. 

c) In particular, a) and b) holds for Galois pairs.
\end{thm}
Part c) follows from a) and b) using the fact that any Galois involution
satisfies $(\star)$ (see Theorem \ref{thmC.1:Any-Galois-involution}). 

Replacing $\underline{G}(\mathbb{R})$ with $\mathrm{ker}(\omega_{G})$,
the following corollary now easily follows:
\begin{cor}
(Conjecture 3 of \cite{Pra}) Let $\underline{G}$ be a connected
reductive algebraic group defined over $\mathbb{R}$ and $\omega_{G}:\underline{G}(\mathbb{R})\longrightarrow\mathbb{Z}/2\mathbb{Z}$
be the character defined in \cite[Section 8]{Pra}. Then for any $\pi\in\mathcal{SAF}_{\mathrm{Irr}}(\underline{G}(\mathbb{C}))$
that is $(\omega_{G},\underline{G}(\mathbb{R}))$-distinguished we
have $\widetilde{\pi}\simeq\pi^{\theta}$.
\end{cor}
In particular, we see that the symmetry condition $\tilde{\pi}\simeq\pi^{\theta}$
is not sensitive to distinction up to characters with finite image,
such as $\omega_{G}$, although the property of distinction of representations
is itself sensitive to such characters (e.g, the Steinberg representation
of $\mathrm{PGL}_{2}(\mathbb{C})$ is not distinguished but only $\omega_{G}$-distinguished).

\subsection{Structure of the paper}

We prove Theorem \ref{thm 1.3} in several steps. In the first step
we translate the problem from the language of representations to the
language of invariant distributions. We use the complex version of
the Langlands' classification to present a given representation $\pi$
as the unique quotient of $\mathrm{Ind}_{B}^{G}(\chi)$, where $\chi$
is some dominant character (see Theorem \ref{thm3.1}), and to deduce
the following:
\begin{thm*}
(Step 1-Corollary \ref{Corollary 3.7}) Let $\pi\in\mathcal{SAF}_{\mathrm{Irr}}(G)$.
Then $\mathrm{dim}_{\mathbb{C}}\left(\pi^{*}\right)^{H}\leq\mathrm{dim}_{\mathbb{C}}\mathcal{S}^{*}(G/H)^{B,\chi\cdot\delta_{B}^{-1/2}}$,
where $\chi$ is the Langlands parameter of $\pi$ (see Theorem \ref{thm3.1}),
and $\mathcal{S}^{*}(X)$ denotes the space of Schwartz distributions
on a manifold $X$ (see Section \ref{sub:Some tools from...}).
\end{thm*}
The next step is to use some geometric results about $B$-orbits on
$G/H$ (see Section \ref{sec:On the geometry of}, \cite{Sp85} and
\cite{HW93}), together with some tools from the theory of distributions
(see Section \ref{sub:Some tools from...}), to reduce the problem
to a question about a single $B$-orbit. More precisely, we prove
the following theorem:
\begin{thm*}
(Step 2- Theorem \ref{Theorem 5.2}) If $\mathcal{S}^{*}(G/H)^{B,\chi\cdot\delta_{B}^{-1/2}}\neq0$
then there exists $n\in N_{G}(T)$ such that $\theta(n)=n^{-1}$ and:
\[
\mathrm{dim}_{\mathbb{C}}\left(\mathrm{Sym}(N_{B(n),n}^{Q'})\otimes_{\mathbb{R}}\mathbb{C})^{T^{\theta_{n}},\chi}\right)>0,
\]
where $Q'=\{g\in G|\theta(g)=g^{\text{\textminus}1}\}$, $B(n)$ is
the $B$ orbit of $n\in Q'\cap N_{G}(T)$, and $T^{\theta_{n}}$ is
the fixed points subgroup of $T$ under $\theta_{n}$.
\end{thm*}
Finally, by direct calculations, we decompose $\mathrm{Sym}(N_{B(n),n}^{Q'})\otimes_{R}\mathbb{C})^{T^{\theta_{n}},\chi}$
into eigenspaces of $T^{\theta_{n}}$ that correspond to some negative
weights and use the fact that $\chi$ is dominant to deduce the following
theorem, which implies Theorem \ref{thm 1.3} a) and b). 
\begin{thm*}
(Step 3- Theorem \ref{thm6.2}) If there exists $n\in N_{G}(T)$ such
that $\theta(n)=n^{-1}$ and:
\[
\mathrm{dim}_{\mathbb{C}}\left(\mathrm{Sym}(N_{B(n),n}^{Q'})\otimes_{\mathbb{R}}\mathbb{C})^{T^{\theta_{n}},\chi}\right)>0,
\]
 then $\mathrm{dim}_{\mathbb{C}}\left(\mathrm{Sym}(N_{B(n),n}^{Q'})\otimes_{R}\mathbb{C})^{T_{n},\chi}\right)=1$
and $\tilde{\pi}\simeq\theta\circ\pi$.
\end{thm*}
In Section 2 we present some basic definitions and tools from the
theory of distributions. In Section 3 we introduce the complex version
of Langlands' classification and some corollaries, including the proof
of the first step (Corollary \ref{Corollary 3.7}). In Section 4 we
present several results from \cite{Sp85} and \cite{HW93} on $B$-orbits
on $G/H$. In Section 5 we make several reductions using the tools
in Section \ref{sub:Some tools from...} and prove Theorem \ref{Theorem 5.2}-
the second step. In Section 6 we prove the third step and finish the
proof of Theorem 1.3. 
\begin{acknowledgement*}
I would like to thank my advisors \textbf{Avraham Aizenbud} and \textbf{Dmitry
Gourevitch} for presenting me with this problem, teaching and helping
me in this work. I hold many thanks to \textbf{Dipendra Prasad }for
introducing me with his work and for a very fruitful discussion and
enlightening remarks. I also thank \textbf{Erez Lapid, Yotam Hendel},
\textbf{Shachar Carmeli} and \textbf{Arie Levit} for helpful discussions.
I was partially supported by ISF grant 687/13, ISF grant 756/12, ERC
StG 637912 and by Minerva foundation grant.
\end{acknowledgement*}

\section{Preliminaries}

\subsection{Basic definitions}

A $\emph{reductive}$ algebraic group is an algebraic group $\underline{G}$
over an algebraically closed field such that the unipotent radical
of $\underline{G}$ is trivial. A $\emph{real reductive group}$ $G$
is the real points $G=\underline{G}(\mathbb{R})$ of an algebraic
group $\underline{G}$ defined over $\mathbb{R}$, where $\underline{G}(\mathbb{C})$
is reductive. The group $\underline{G}(\mathbb{R})$ inherits the
natural topology from $\mathbb{R}^{N}$ and this gives $\underline{G}(\mathbb{R})$
a structure of a Lie group. 

Let $G=\underline{G}(\mathbb{R})$ be a real reductive group and $H$
a closed subgroup. A pair $(G,H)$ is called a $real\,\emph{symmetric\,pair}$
if there exists an involution $\theta$ of $\underline{G}$ such that
$H$ is an open subgroup of $\underline{G}^{\theta}(\mathbb{R})$
of $\theta$-invariant elements. 

A $\emph{Fréchet representation}$ $(\pi,V)$ of $G$ over $\mathbb{C}$
is a morphism of groups $\pi:G\rightarrow Aut_{\mathbb{C}}(V)$ ,
such that $V$ is a Fréchet $\mathbb{C}-$vector space, and that $G\times V\rightarrow V$
is continuous. A representation $(\pi,V)$ of $G$ over $\mathbb{C}$
is called $\emph{irreducible}$ if $V\neq0$ and if $V$ has no proper,
closed $G$-invariant subspace. A map $f:G\rightarrow V$ is called
$\emph{smooth}$ if it has continuous partial derivatives of all orders,
where the derivatives and their continuity are with respect to the
Lie structure of $G$. A vector $v\in V$ is called $\emph{smooth}$
if its orbit map $c_{v}:G\rightarrow V$ defined by $c_{v}(g)=\pi(g)v$
is a smooth map. 

A representation $(\pi,V)$ of $G$ over $\mathbb{C}$ is called:
\begin{itemize}
\item $\emph{Smooth}$, if every $v\in V$ is smooth.
\item $\emph{Admissible}$, if for every maximal compact subgroup $K$ and
every irreducible representation $\rho$ of $K$, $\rho$ has finite
multiplicity in $\pi|_{K}$.
\end{itemize}
If $(\pi,V)$ is a smooth representation, the dual representation
$(\pi^{*},V^{*})$ is defined by $V^{*}:=\mathrm{Hom}_{\mathbb{C}}(V,\mathbb{C})$
with the usual left $G$-action $g(f)(v)=f(g^{-1}v)$. It may not
be smooth, and also $V^{*}$ is not necessarily a Fréchet space. Thus,
we define $V^{*\infty}:=$$C_{c}^{\infty}(G)(V^{*})$ where $C_{c}^{\infty}(G)$
acts on $V^{*}$ by $f(\varphi)(v)=\int_{G}f(g)\cdot\langle\pi^{*}(g)(\varphi),v\rangle dg$,
where $f\in C_{c}^{\infty}(G)$ and $\varphi\in V^{*}$. This space
is called the $\emph{Garding space}$ and it is a Fréchet space consisting
of smooth vectors, so if we denote $\tilde{V}:=V^{*\infty}$, this
defines a smooth representation $(\tilde{\pi},\tilde{V})$ called
the $\emph{contragredient representation}$. A smooth representation
$(\pi,V)$ of $G$ such that $\left(\pi^{*}\right)^{H}\neq0$ for
a subgroup $H$, is called $\emph{H-distinguished}$. 

Assume that $G$ is real reductive group and fix a maximal compact
$K$. Fix a faithful algebraic representation $\rho:G\hookrightarrow\mathrm{GL}_{n}(\mathbb{R})$
and define $\left\Vert g\right\Vert :=\mathrm{Tr}(\rho(g)\cdot\rho(g)^{t})+\mathrm{Tr}(\rho(g^{-1}),\rho(g^{-1})^{t})$.
Let $(\pi,V)$ be a smooth Fréchet representation and denote by $\left\{ P_{i}\right\} _{i\in I}$
the continuous semi-norms that define the Fréchet structure of $V$.
We say that $(\pi,V)$ is $\emph{of moderate growth}$ if for any
semi-norm $P_{i}$ there exists $P_{j}$ and $k(i)\in\mathbb{N}_{>0}$
such that:
\[
P_{i}(\pi(g)v)\leq\left\Vert g\right\Vert ^{k(i)}\cdot P_{j}(v)
\]
for all $g\in G$ and $v\in V$ (see \cite[2.3]{BK14} for more details).
A vector $v\in(\pi,V)$ is called a $\emph{K-finite vector}$ if $\pi(K)v$
lies in a finite dimensional subspace of $V$. The space of $K$-finite
vectors of $V$ is denoted $V_{K}$. $(\pi,V)$ is called $\emph{finitely generated}$
if $V_{K}$ is a finitely generated $\mathcal{U}(\mathfrak{g})$-module,
where $\mathfrak{g}$ is the Lie algebra of $G$ and $\mathcal{U}(\mathfrak{g})$
is its universal enveloping algebra. 

Putting everything together, we denote by $\mathcal{SAF}(G)$ the
category of all finitely generated, smooth, admissible, moderate growth,
Fréchet representations of $G$ with continuous linear $G$-maps as
morphisms.

\subsection{\label{sub:Some tools from...}Some tools from the theory of distributions}

We will work with the notation and tools presented in \cite[Appendix B]{ALOF12}.
For the convenience of the reader, we present notation and reformulations
of some of the theorems that appear in \cite[Appendix B]{ALOF12},
in a version better suited for this work. The results in this section
holds for $\emph{Nash manifolds}$, i.e smooth, semi-algebraic varieties.
As any smooth algebraic variety can be endowed with a natural structure
of a Nash manifold, we can apply the results presented in this section
for smooth\textbf{ }algebraic varieties . Let $G$ be an arbitrary
group. 
\begin{itemize}
\item For any $G$-set $X$ and a point $x\in X$, we denote by $G(x)$
the $G$-orbit of $x$ and by $G_{x}$ the stabilizer of $x$.
\item For any representation of $G$ on a vector space $V$ and a character
$\chi$ of $G$, we denote by $V^{G,\chi}$ the subspace of $(G,\chi)$-equivariant
vectors in $V$.
\item Given manifolds $L\subseteq M$, we denote by $N_{L}^{M}:=(T_{M}|_{L})/T_{L}$
the normal bundle to $L$ in $M$ and by $CN_{L}^{M}:=(N_{L}^{M})^{*}$
the conormal bundle. For any point $y\in L$, we denote by $N_{L,y}^{M}$
the normal space to $L$ in $M$ at the point $y$ and by $CN_{L,y}^{M}$
the conormal space to $L$ in $M$ at the point $y$.
\item The symmetric algebra of a vector space $V$ will be denoted by $\mathrm{Sym}(V)=\oplus_{k\ge0}\mathrm{Sym^{k}}(V)$. 
\item The Fréchet space of Schwartz functions on a Nash manifold $X$ will
be denoted by $\mathcal{S}(X)$ and the dual space of Schwartz distributions
will be denoted by $\mathcal{S}^{*}(X)$ .
\item For any Nash vector bundle $E$ over $X$ we denote by $\mathcal{S}(X,E)$
the space of Schwartz sections of $E$ and by $\mathcal{S}^{*}(X,E)$
its dual space. 
\item For a closed subset $Z$ of a smooth manifold $X$ we set $\mathcal{S}_{Z}^{*}(X):=\{\xi\in\mathcal{S}^{*}(X):\,\mathrm{Supp}(\xi)\subseteq Z\}$.\end{itemize}
\begin{prop}
\label{prop 2.1}If $U$ is an open submanifold of $X$ then we have
the following exact sequence:

\[
0\longrightarrow\mathcal{S}_{X\backslash U}^{*}(X)\overset{i}{\longrightarrow}\mathcal{S}^{*}(X)\overset{ext^{*}}{\longrightarrow}\mathcal{S}^{*}(U)\longrightarrow0
\]

where $i:\mathcal{S}_{X\backslash U}^{*}(X)\hookrightarrow\mathcal{S}^{*}(X)$
is the natural inclusion and $ext^{*}:\mathcal{S}^{*}(X)\longrightarrow\mathcal{S}^{*}(U)$
is the dual map to extension by zero. \end{prop}
\begin{defn*}
If $X$ is a smooth manifold and $G$ acts on $X$, then $X=\overset{l}{\underset{i=1}{\cup}}X_{i}$
is called a $G$-$\emph{invariant stratification}$ if all sets $X_{i}$
are $G$-invariant and $\overset{j}{\underset{i=1}{\cup}}X_{i}$ is
open in $X$ for any $1\leq j\leq l$.\end{defn*}
\begin{prop}
\label{prop2.2}Let $G$ be a Nash group that acts on a Nash manifold
$X$. Let $X=\overset{l}{\underset{i=1}{\cup}}X_{i}$ be a $G$-invariant
stratification . Let $\chi$ be a character of $G$. Then:
\[
\mathrm{dim}_{\mathbb{C}}(\mathcal{S}^{*}(X)^{G,\chi})\leq\sum_{i=1}^{l}\sum_{k=0}^{\infty}\mathrm{dim}_{\mathbb{C}}(\mathcal{S}^{*}(X_{i},\mathrm{Sym^{k}}(CN_{X_{i}}^{X}))^{G,\chi}).
\]

\end{prop}

\begin{prop}
\label{Prop2.3}Let $\underline{G}$ be an algebraic groups and $\underline{H}$
a closed subgroup. Set $G=\underline{G}(\mathbb{R}),\,H=\underline{H}(\mathbb{R})$
. Then $\mathcal{S}^{*}(G)^{H}\simeq\mathcal{S}^{*}(G/H)$.
\end{prop}
Proposition \ref{prop2.2} gives us a way to reduce from asking if
a manifold $X$ has a non-trivial $(G,\chi)$ equivariant distribution
to asking a similar question on a $G$-orbit. The next tool can help
us to further reduce the problem: 
\begin{thm}
\label{thm:(Frobenius-descent).-Let}(Frobenius descent). Let a Nash
group $G$ act transitively on a Nash manifold $Z$. Let $\varphi:X\rightarrow Z$
be a $G$-equivariant Nash map. Let $z\in Z$ and let $X_{z}$ be
the fiber of $z$. Let $\chi$ be a tempered character of $G$ (see
\cite[Definition 5.1.1]{AG08}). Then $\mathcal{S}^{*}(X)^{G,\chi}$
is canonically isomorphic to $\mathcal{S}^{*}(X_{z})^{G_{z},\chi\delta_{G_{z}}^{-1}\delta_{G}|_{G_{z}}}$,
where $G_{z}$ is the stabilizer of $z$ in $G$. Moreover, for any
$G$-equivariant bundle $E$ on $X$, the space $\mathcal{S}^{*}(X,E)^{G,\chi}$
is canonically isomorphic to $\mathcal{S}^{\ast}(X_{z},E|_{X_{z}})^{G_{z},\chi\delta_{G_{z}}^{\text{\textminus}1}\delta_{G}|_{G_{z}}}$. \end{thm}
\begin{cor}
\label{cor:Let-a-Nash}Let a Nash group $G$ acts on a Nash manifold
$X$. Let $X=\overset{l}{\underset{i=1}{\cup}}X_{i}$ be a $G$-invariant
stratification. For each $i$, We have a canonical isomorphism: 
\[
\mathcal{S}^{*}(X_{i},\mathrm{Sym^{k}}(CN_{X_{i}}^{X}))^{G,\chi}\simeq\mathcal{S}^{*}(\{x\},\mathrm{Sym^{k}}(CN_{X_{i},x}^{X}))^{G_{x},\chi\delta_{G_{x}}^{\text{\textminus}1}\delta_{G}|_{G_{x}}}.
\]
\end{cor}
\begin{proof}
Choose $Z=X_{i}$, $Id:X_{i}\longrightarrow X_{i}$ and $z=x$ and
$X_{z}=\{x\}$ and apply Theorem \ref{thm:(Frobenius-descent).-Let}. \end{proof}
\begin{cor}
\label{cor 2.6}Let a Nash group $G$ acts on a Nash manifold $X$.
Let $X=\overset{l}{\underset{i=1}{\cup}}X_{i}$ be a $G$-invariant
stratification. Then:
\[
\mathrm{dim}_{\mathbb{C}}(\mathcal{S}^{*}(X)^{G,\chi})\leq\sum_{i=1}^{l}\sum_{k=0}^{\infty}\mathrm{dim}_{\mathbb{C}}\left((\mathrm{Sym^{k}}(N_{X_{i},x}^{X})\otimes_{\mathbb{R}}\mathbb{C})^{G_{x},\chi\delta_{G_{x}}^{\text{\textminus}1}\delta_{G}|_{G_{x}}}\right).
\]
\end{cor}
\begin{proof}
Combine Proposition \ref{prop2.2} with Corollary \ref{cor:Let-a-Nash}
and use the isomorphism
\[
\mathrm{Sym^{k}}(N_{X_{i},x}^{X})\otimes_{\mathbb{R}}\mathbb{C}\simeq\mathcal{S}^{*}(\{x\},\mathrm{Sym^{k}}(CN_{X_{i},x}^{X})).
\]

\end{proof}

\section{\label{sec:The-complex-version}The complex version of the Langlands
classification }

The main goal of this section is to translate the property of being
$H$-distinguished to the language of invariant distributions using
the Langlands classification for complex reductive groups.

We assume that $G$ is a complex connected reductive group. Let $\mathfrak{g}$
be its Lie algebra, $K$ be a maximal compact subgroup of $G$ with
Lie algebra $\mathfrak{k}$, $\tau$ the Cartan involution fixing
$K$, and $\mathfrak{p}$ the $-1$ eigenspace of $\tau$. Let $\mathfrak{a}_{\mathfrak{p}}$
be a maximal abelian subspace of $\mathfrak{p}$, with the corresponding
analytic subgroup $A_{\mathfrak{p}}$. Fix $A_{\mathfrak{p}}\subseteq A\subseteq T\subseteq B$,
where $T$ is a maximal torus, $A$ is a maximal $\mathbb{R}$-split
torus, and $B$ is a Borel subgroup that contains $T$.

The work of Zhelobenko and Naimark \cite{ZN66} and Langlands \cite{Lan73}
provides the following known classification of representations $\pi\in\mathcal{SAF}_{\mathrm{Irr}}(G)$
for a complex reductive group $G$:
\begin{thm}
\label{thm3.1}Let $G$ be a connected complex reductive group and
let $\pi\in\mathcal{SAF}_{Irr}(G)$ . Then $\pi$ is the unique irreducible
quotient of $\mathrm{Ind}_{B}^{G}(\chi)$, where
\[
\mathrm{Ind}_{B}^{G}(\chi)=\{f\in C^{\infty}(G)|\forall b\in B,\,f(bg)=\chi(b)\cdot\delta_{B}^{1/2}(b)\cdot f(g)\}=C^{\infty}(G)^{B,\chi^{-1}\delta_{B}^{-1/2}},
\]
 and $\chi$ is a dominant character of $T$.
\end{thm}
In Appendix \ref{subA1Proof-of-Theorem} we deduce Theorem \ref{thm3.1}
from the Langlands classification (see Theorem \ref{thm:(The-Langlands'-Classification)}),
although historically, the Langlands classification appeared after
the work of Zhelobenko and Naimark, who dealt with the case of a complex
group.

Set $W(G,A_{\mathfrak{p}})$ and $W(G,T)$ to be the Weyl groups corresponding
to the root systems $\Sigma(G,A_{\mathfrak{p}})$ and $\Sigma(G,T)$.
Note that the Weyl group $W(G,T)\simeq N_{G}(T)/T$ acts on $T$ by
$w.t:=ntn^{-1}$, where $n\in N_{G}(T)$ and $w=nT$. Similarly $w.\chi(t):=\chi(n^{-1}tn)$. 
\begin{prop}
\label{prop:Let--and}Let $\pi_{1}=\mathrm{Ind}_{B}^{G}(\chi)$ and
$\pi_{2}=\mathrm{Ind}_{B}^{G}(\chi')$ in $\mathcal{SAF}(G)$, be
unitary principal series representations of $G$. Then there exists
an intertwining operator $L:\pi_{1}\longrightarrow\pi_{2}$ if and
only if there exists $w\in W(G,A_{\mathfrak{p}})$ such that $w.\chi=\chi'$. \end{prop}
\begin{proof}
Follows from \cite[VII.4]{Kn01} with a slight modification to the
category $\mathcal{SAF}(G)$. 
\end{proof}
This implies the following:
\begin{cor}
\label{cor:Let--and be unitary}Let $\mathrm{Ind}_{B}^{G}(\chi)$
and $\mathrm{Ind}_{B}^{G}(\chi')$ be irreducible unitary principal
series representations of $G$. Then they are isomorphic if and only
if there exists $w\in W(G,A_{\mathfrak{p}})$ such that $w.\chi=\chi'$.
\end{cor}
In Appendix \ref{sub:Proof-of-Corollary 3.4} we prove the following
corollary:
\begin{cor}
\label{cor 3.4}Let $\pi_{1},\pi_{2}\in\mathcal{SAF}_{\mathrm{Irr}}(G)$
where $\pi_{1}$ is the unique quotient of $\mathrm{Ind}_{B}^{G}(\chi_{1})$
and $\pi_{2}$ is the unique quotient of $\mathrm{Ind}_{B}^{G}(\chi_{2})$.
Then $\pi_{1}\simeq\pi_{2}$ if and only if there exists $w\in W(G,A_{\mathfrak{p}})$
such that $w.\chi_{1}=\chi_{2}$. 
\end{cor}
Theorem \ref{thm:If--is unique quotient} and Corollary \ref{cor 3.6}
are needed for Section 6:
\begin{thm}
\label{thm:If--is unique quotient} If $\pi$ is the unique irreducible
quotient of $\mathrm{Ind}_{B}^{G}(\chi)$ then:

1) $\widetilde{\pi}$ is the unique irreducible quotient of $\mathrm{Ind}_{B}^{G}(w_{0}.(\chi^{-1}))$,
where $w_{0}\in W(G,T)$ is the longest element in the Weyl group. 

2) Let $\theta:G\longrightarrow G$ be an involution and assume that
$T$ is $\theta$-stable. Then $\pi^{\theta}$ is the unique irreducible
quotient of $\mathrm{Ind}_{B}^{G}(w'.\theta(\chi))$ for some $w'\in W(G,T)$
such that $w'.\theta(\chi)$ is dominant. \end{thm}
\begin{proof}
1) The non degenerate pairing between $\mathrm{Ind}_{B}^{G}(\chi^{-1})$
and $\mathrm{Ind}_{B}^{G}(\chi)$ by $\langle f,h\rangle:=\int_{B\backslash G}f\cdot h$
induces an isomorphism $\widetilde{\mathrm{Ind}_{B}^{G}(\chi)}\simeq\mathrm{Ind}_{B}^{G}(\chi^{-1})$.
As $\pi$ is the unique irreducible quotient of $\mathrm{Ind}_{B}^{G}(\chi)$,
$\widetilde{\pi}$ is the unique irreducible subrepresentation of
$\mathrm{Ind}_{B}^{G}(\chi^{-1})=\mathrm{Ind}_{w_{0}.B^{-}}^{G}(\chi^{-1}))$,
where $B{}^{-}$ is the opposite Borel to $B$. By \cite[VIII.15, Theorem 8.54]{Kn01},
$\widetilde{\pi}$ is the unique quotient of $\mathrm{Ind}_{w_{0}.B}^{G}(\chi^{-1}))\simeq\mathrm{Ind}_{B}^{G}(w_{0}.(\chi^{-1}))$.
As $\chi^{-1}$ is anti-dominant we have that $w_{0}.(\chi^{-1})$
is dominant.

2) Notice that $\theta\circ\mathrm{Ind}_{B}^{G}(\chi)$ is naturally
isomorphic to $\mathrm{Ind}_{\theta(B)}^{G}(\theta(\chi))$ by $f\longrightarrow f^{\theta}$,
where $f^{\theta}(g):=f(\theta(g))$. This implies that $\pi^{\theta}$
is the unique irreducible quotient of $\mathrm{Ind}_{\theta(B)}^{G}\left(\theta(\chi))\right)$.
Notice that $\theta(B)$ is a Borel subgroup of $G$ that contains
$\theta(T)=T$ and as $W(G,T)$ act transitively on the Borel subgroups
of $G$ that contains $T$, there exists $w'\in W(G,T)$ such that
$w'.B=\theta(B)$. Thus $\pi^{\theta}$ is the unique irreducible
quotient of $\mathrm{Ind}_{B}^{G}\left(w'.\theta(\chi))\right)$ and
clearly $w'.\theta(\chi)$ is dominant.\end{proof}
\begin{cor}
\label{cor 3.6}Assume that there exists $w\in W(G,T)$ such that
$\chi^{-1}=w.\left(\theta(\chi)\right)$ and $w\circ\theta=\theta\circ w^{-1}$.
Then $\widetilde{\pi}^{\theta}\simeq\pi$. \end{cor}
\begin{proof}
If $\chi^{-1}=w.\theta(\chi)=\theta.w^{-1}(\chi)$ then $\theta(\chi^{-1})=w^{-1}.\chi$.
By Theorem \ref{thm:If--is unique quotient}, there exists $w'\in W(G,T)$
such that $\widetilde{\pi}^{\theta}$ is the unique irreducible quotient
of $\mathrm{Ind}_{B}^{G}(w'.\theta(\chi^{-1}))$. By Corollary \ref{cor 3.4},
$\widetilde{\pi}^{\theta}\simeq\pi$ as $w'.\theta(\chi^{-1})=w'w^{-1}.\chi$.
\end{proof}
We now use Theorem \ref{thm3.1} to translate our question into the
language of distributions. The quotient map $\mathrm{Ind}_{B}^{G}(\chi)\longrightarrow\pi$
induces an injection $\left(\pi^{*}\right)^{H}\hookrightarrow\left(\mathrm{Ind}_{B}^{G}(\chi)^{*}\right)^{H}$
where $\chi$ is the Langlands parameter of $\pi$ from Theorem \ref{thm3.1}.
Thus, it is enough to study $H$-distinguished representations of
the form $\sigma=\mathrm{Ind}_{B}^{G}\chi$. Since $B\backslash G$
is compact, it follows from (e.g. \cite[Lemma 2.9]{GSS15}) that
\[
\sigma^{*}=\left(\left(C^{\infty}(G)^{B,\chi^{-1}\delta_{B}^{-1/2}}\right)^{*}\right)^{H}\simeq\mathcal{S}^{*}(G)^{B\times H,\chi\cdot\delta_{B}^{-1/2}\times1}.
\]
 Thus, together with Proposition \ref{Prop2.3}, we obtain the following
corollary:
\begin{cor}
\label{Corollary 3.7} We have that 
\[
\mathrm{dim}_{\mathbb{C}}\left(\pi^{*}\right)^{H}\leq\mathrm{dim}_{\mathbb{C}}\mathcal{S}^{*}(G/H)^{B,\chi\cdot\delta_{B}^{-1/2}}.
\]

\end{cor}
This proves Step 1 (see Introduction). We now move towards Step 2,
but before we prove it, we need to understand the geometry of $B$-orbits
on $G/H$. This brings us to the next section.

\section{\label{sec:On the geometry of}On the geometry of $B$-orbits on
$G/H$}

The study of the geometry of the minimal parabolic orbits on a symmetric
space $G/H$ is well developed. The case of complex symmetric spaces
was studied by Springer \cite{Sp85} and the generalization for any
local field was done in \cite{HW93}. In this section we use results
mostly from \cite{HW93}, as we are interested in real symmetric spaces.
In order to apply the results in \cite{HW93} to our setting we need
to make some modifications: 
\begin{enumerate}
\item \cite{HW93} studies the orbit decomposition of $\underline{B}(\mathbb{R})$
on $\underline{G}(\mathbb{R})/\underline{H}(\mathbb{R})$, where $\underline{H}$
is an (Zarisky) open subgroup of $\underline{G}^{\theta}$. We need
to generalize to the case when $H$ is an open subgroup (in the Hausdorff
topology) of $\underline{G}^{\theta}(\mathbb{R})$. 
\item We want to show that there exists a stratification of $G/H$ by $B$-orbits. 
\end{enumerate}
Let $\underline{G}$ be a connected reductive algebraic group defined
over $\mathbb{R}$, $\theta$ be an $\mathbb{R}$-involution of $\underline{G}$,
$\underline{G}^{\theta}$ be the $\theta$-fixed points of $\underline{G}$
and $\underline{P}$ be a minimal parabolic subgroup. Let $\underline{U}=R_{u}(\underline{P})$
be the unipotent radical of $\underline{P}$. We recall the following
known result:
\begin{prop}
\label{prop 4.1} \cite[Lemma 2.4]{HW93} Every minimal parabolic
$\mathbb{R}$-subgroup $\underline{P}$ of $\underline{G}$ contains
a $\theta$-stable maximal $\mathbb{R}$-split torus $\underline{A}$
of $\underline{P}$, unique up to conjugation by an element of $\left(\underline{G}^{\theta}\cap\underline{U}\right)(\mathbb{R})$.
\end{prop}
Let $\underline{A}$ be a $\theta$-stable maximal $\mathbb{R}$-split
torus of $\underline{P}$ . Given $g,x\in\underline{G}$, let $\rho$
be the $\emph{twisted action}$ associated to $\theta$, given by
$\rho(g,x)=gx\theta(g)^{-1}$. Denote $\underline{Q}=\{g^{-1}\theta(g)|g\in\underline{G}\}$,
$\underline{Q}'=\{g\in G|\theta(g)=g^{-1}\}$ and note that $\underline{Q}$
and $\underline{Q}'$ are invariant under $\rho$. Define a morphism
$\tau:\underline{G}\rightarrow\underline{G}$ by $\tau(x)=x\theta(x^{-1})$.
We have the following facts (see \cite[Section 6]{HW93}): 

{*} $\underline{Q}=\tau(\underline{G})$ is a connected closed subvariety
of $\underline{G}$. 

{*} $\underline{Q}'$ is a finite union of twisted $G$-orbits, and
each orbit is closed. In particular, $\underline{Q}$ is a connected
component in $\underline{Q}'$.

Define $Q=\tau(\underline{G}(\mathbb{R}))=\tau(G)$, $Q'=\underline{Q}'(\mathbb{R})$
and notice that $\tau$ induces an isomorphism $\widetilde{\tau}:G/G^{\theta}\simeq Q$.
Write $U=\underline{U}(\mathbb{R})$, $P=\underline{P}(\mathbb{R})$,
$A=\underline{A}(\mathbb{R})$ and $N=N_{G}(A)$. 
\begin{prop}
\cite[Proposition 6.6]{HW93} If $g\in Q'$, then there exists $x\in U$
such that $xg\theta(x)^{-1}\in N$. Therefore, twisted $P$ orbits
on $Q'$ are represented by $n\in N\cap Q'$. 
\end{prop}
Consider $\tau^{-1}(N)=\{g\in G|g^{-1}\theta(g)\in N\}$. The group
$G^{\theta}\times C_{G}(A)$ acts on $\tau^{\text{\textminus}1}(N)$
by $(x,y).z=xzy^{-1}$. Let $V$ denote the set of orbits of $G^{\theta}\times C_{G}(A)$
in $\tau^{\text{\textminus}1}(N)$ . We identify $V$ with a fixed
set of representatives of the orbits in $\tau^{\text{\textminus}1}(N)$. 
\begin{prop}
\cite[Proposition 6.8]{HW93}$G$ is the disjoint union of the double
cosets $PvG^{\theta}$, $v\in V$. \end{prop}
\begin{cor}
$\widetilde{\tau}:G/G^{\theta}\longrightarrow Q$ induces a bijection
between $P$- orbits on $G/G^{\theta}$ under the left action with
$P$-orbits in $Q$ under the twisted action $\rho$. The double cosets
$PxG^{\theta}$ are represented by elements $x\in\tau^{-1}(N)$, and
the $P$ orbits in $Q$ are represented by elements $n\in N\cap Q$.\end{cor}
\begin{proof}
$\widetilde{\tau}$ is an isomorphism, so the bijection between $P$-orbits
is clear. As the twisted $P$-orbits in $Q$ are represented by $n\in N$,
the double cosets $PxG^{\theta}$ are represented by $x\in\tau^{-1}(N)$. \end{proof}
\begin{prop}
\cite[Proposition 6.15]{HW93}$Q'$ and $Q$ has only finitely many
twisted $P$-orbits.
\end{prop}
Combining the above results, we have the following corollary:
\begin{cor}
\label{cor4.6}$Q'$ and $Q$ has finitely many twisted $P$-orbits,
each is of the form $P.n$ for $n\in N$ such that $\theta(n)=n^{-1}$. 
\end{cor}
The following Theorem \ref{Theorem 4.7} and Corollary \ref{cor 4.8}
are used in Section 5, and they are essential for applying the tools
from the Section 2 to the space $\mathcal{S}^{*}(G/H)^{B,\chi\cdot\delta_{B}^{-1/2}}$.
\begin{thm}
\label{Theorem 4.7}There exists a stratification $Q=\bigcup{}_{k=1}^{m}P(n_{k})$
by $P$-orbits, each contains some $n_{k}\in N_{G}(T)\cap Q$. \end{thm}
\begin{proof}
The following observations shows that it is enough to show that $Q'$
has a such a stratification. 

{*} By the argument above, we have that $Q'$ decomposes into a finite
number of $P$-orbits, and hence also to a finite number of $G$-orbits.
This implies that also $\underline{Q}(\mathbb{R})$ has finite number
of $G$ orbits.

{*} The map $\tau:\underline{G}\longrightarrow\underline{Q}$ is a
submersion at $e$ \cite[Lemma 6.13]{HW93}, hence $\tau(\underline{G}(\mathbb{R}))=Q$
is open in $\underline{Q}(\mathbb{R})$. Note that $\tau:\underline{G}(\mathbb{R})\longrightarrow\underline{Q}(\mathbb{R})$
is the orbit map of $e\in\underline{Q}(\mathbb{R})$. In a similar
argument one can show that each $G$-orbit in $\underline{Q}(\mathbb{R})$
is open, and since there are finite such orbits, they are also closed.
Thus $Q$ is open and closed in $\underline{Q}(\mathbb{R})$.

{*} Since $\underline{Q}(\mathbb{R})$ is open and closed in $Q'$,
we have that $Q$ is open and closed in $Q'$ and hence a stratification
on $Q'$ restrict to a stratification on $Q$.

In order to show that $Q'$ has such a stratification we use the following
steps: 

{*} By \cite[Proposition 6.15]{HW93}, $\underline{Q}'$ has finitely
many twisted $\underline{P}$-orbits. By induction on the number of
these orbits, and the fact that any $\underline{G}$-variety has a
(Zarisky) closed orbit, it is easy to see that $\underline{Q}'$ has
a stratification $\underline{Q}'=\bigcup{}_{k=1}^{t}O_{\underline{P}}^{k}$
of $\underline{P}$-orbits $O_{\underline{P}}^{k}$.

{*} Write $Q'=\bigcup{}_{k=1}^{t}\left(O_{\underline{P}}^{k}\cap Q'\right)$.
It is a known fact, that for each $\underline{P}$-orbit $O_{\underline{P}}^{k}$,
we have that $O_{\underline{P}}^{k}\cap Q'$ is a finite union of
$P$-orbits $O_{j}$, $O_{\underline{P}}^{k}\cap Q'=\bigsqcup_{j=1}^{,m_{k}}O_{j}$,
where each $O_{j}$ is open and closed in $O_{\underline{P}}^{k}\cap Q'$.
This gives the required stratification of $Q'$. \end{proof}
\begin{cor}
\label{cor 4.8} Let $G=\underline{G}(\mathbb{R})$ be a complex connected
reductive group, $H$ be an open subgroup of $G^{\theta}=\underline{G}^{\theta}(\mathbb{R})$,
and $B=\underline{B}(\mathbb{R})$ a Borel subgroup of $G$. Then
there exists a finite stratification of $G/H$ by $B$-orbits.\end{cor}
\begin{proof}
Let $\pi:G/H\longrightarrow G/G^{\theta}$ be the natural quotient
map and note that $\pi$ is a cover map, with fibers of size $\left|G^{\theta}:H\right|$.
Let $G/G^{\theta}=\bigsqcup_{i=1}^{m}X_{i}$ be a $B$-stratification
of $G/G^{\theta}$ and assume that $X_{i}$ is the $B$-orbit of $\{x_{i}G^{\theta}\}\in G/G^{\theta}$.
Write $G^{\theta}=\bigsqcup h_{i}H$, where $h_{i}$ are representatives
of the coset space $G^{\theta}/H$. Then $\pi^{-1}(X_{i})$ is a $B$-invariant
set that contains at most $\left|G^{\theta}:H\right|$ $B$-orbits
(the orbits of $\pi^{-1}(x_{i}G^{\theta})$). Consider the open orbit
$B(\{x_{1}G^{\theta}\})$, and notice that $\pi^{-1}(B(\{x_{1}G^{\theta}\}))$
is a $B$-invariant open set, and since $\pi$ is a local diffeomoerphism,
it follows that each $B$-orbit in $\pi^{-1}(B(\{x_{1}G^{\theta}\}))$
is open in $G/H$. By induction, we can now apply the same argument
for the spaces $\left(G/H\right)-\pi^{-1}(X_{1})$ and $\bigsqcup_{i=2}^{m}X_{i}$
and get a stratification of $G/H$ by $B$-orbits as required. 
\end{proof}

\section{\label{sec:Some-reductions-using}Some reductions using tools from
the theory of distributions}

In this section we assume the setting of Theorem \ref{thm 1.3} and
we use the notation from Section \ref{sec:On the geometry of}. 
\begin{lem}
\label{lemma 5.1} There exists a Borel subgroup $B$, a $\theta$-stable
maximal torus $T$ and a $\theta$-stable maximal $\mathbb{R}$-split
torus $A$, such that $A\subseteq T\subseteq B$.\end{lem}
\begin{proof}
By Proposition \ref{prop 4.1} we can find a $\theta$-stable maximal
$\mathbb{R}$-split torus $A$. As $G$ is complex, $C_{G}(A)=T$
a maximal torus and it is $\theta$-stable. Indeed, let $t\in C_{G}(A)$,
then:
\[
tat^{-1}=a\Longrightarrow\theta(tat^{-1})=\theta(a)\Longrightarrow\theta(t)\theta(a)\theta(t)^{-1}=\theta(a),
\]
so $\theta(t)\in C_{G}(A)$ and $T=C_{G}(A)$ is $\theta$-stable.
We finish by choosing some Borel $B$ that contains $T$. 
\end{proof}
We now fix such $A\subseteq T\subseteq B$. Recall that for any $n\in N_{G}(T)$
such that $\theta(n)=n^{-1}$ we denote by $\theta_{n}$ the involution
of $G$ defined by $\theta_{n}(g)=n\theta(g)n^{-1}$ and by $T^{\theta_{n}}$
(resp. $B^{\theta_{n}}$) the fixed points subgroup of $T$ (resp.
$B$) under $\theta_{n}$. Let $Q,Q',\tau$ and $\widetilde{\tau}$
as in Section \ref{sec:On the geometry of}. The goal in this section
is to prove the following theorem: 
\begin{thm}
\label{Theorem 5.2} If $\mathcal{S}^{*}(G/H)^{B,\chi\cdot\delta_{B}^{-1/2}}\neq0$
then there exists $n\in N_{G}(T)\cap Q$, such that:
\[
\mathrm{dim}_{\mathbb{C}}\left(\mathrm{Sym}(N_{B(n),n}^{Q'})\otimes_{\mathbb{R}}\mathbb{C})^{T^{\theta_{n}},\chi}\right)>0,
\]
where $B(n)$ is the $B$-orbit of $n\in Q'\cap N_{G}(T)$ under the
twisted action.
\end{thm}
Recall that $\widetilde{\tau}:G/G^{\theta}\stackrel{\simeq}{\longrightarrow}Q$
is an isomorphism of $B$-spaces, where $B$ acts on $G/G^{\theta}$
by the left regular action and on $Q$ by the twisted action ($b.q=bq\theta(b)^{-1}$).
By Corollary \ref{cor 4.8} and Theorem \ref{Theorem 4.7}, there
exists a finite stratification of $G/H=\bigcup{}_{i=1}^{m}B(x_{i})$
and $Q=\bigcup{}_{j=1}^{t}B(n_{j})$ by $B$-orbits, where $m=\left|B\backslash G/H\right|$,
$t=\left|B\backslash G/G^{\theta}\right|$ and $n_{j}\in N_{G}(T)$
for any $j$. Let $\pi:G/H\longrightarrow G/G^{\theta}$ be the natural
quotient map. WLOG, we may assume for any $i\in\{1...m\}$ that $\widetilde{\tau}\circ\pi(x_{i})=n_{j_{i}}$
for some $j_{i}\in\{1...t\}$. The following corollary follows from
Proposition \ref{prop2.2}:
\begin{cor}
We have:
\[
\mathrm{dim}_{\mathbb{C}}(\mathcal{S}^{*}(G/H)^{B,\chi\delta_{B}^{-1/2}})\leq\sum_{i=1}^{m}\sum_{k=0}^{\infty}\mathrm{dim}_{\mathbb{C}}(\mathcal{S}^{*}(B(x_{i}),\mathrm{Sym^{k}}(CN_{B(x_{i}),x_{i}}^{G/H}))^{B,\chi\delta_{B}^{-1/2}}).
\]

\end{cor}
We now prove Theorem \ref{Theorem 5.2}: 
\begin{proof}
By Corollary \ref{cor 2.6} we deduce that:
\[
0<\mathrm{dim}_{\mathbb{C}}(\mathcal{S}^{*}(G/H)^{B,\chi\delta_{B}^{-1/2}})\leq\sum_{i=1}^{m}\sum_{k=0}^{\infty}\mathrm{dim}_{\mathbb{C}}((\mathrm{Sym^{k}}(N_{B(x_{i}),x_{i}}^{G/H})\otimes_{\mathbb{R}}\mathbb{C})^{B_{x_{i}},\chi\delta_{B_{x_{i}}}^{\text{\textminus}1}\delta_{B}^{1/2}}).
\]
Since $\pi:G/H\longrightarrow G/G^{\theta}$ is a cover map, it is
a local isomorphism of $B$-spaces. Therefore $\widetilde{\tau}\circ\pi$
is a local isomorphism of $B$-spaces as well, and for each $i$ we
have: 
\[
N_{B(x_{i}),x_{i}}^{G/H}\simeq N_{B(\pi(x_{i})),\pi(x_{i})}^{G/G^{\theta}}\simeq N_{B(n_{j_{i}}),n_{j_{i}}}^{Q}\simeq N_{B(n_{j_{i}}),n_{j_{i}}}^{Q'}
\]
as $B$-representations, where the last equality follows from the
fact that $Q$ is open in $Q'$. We deduce that:
\[
0<\mathrm{dim}_{\mathbb{C}}(\mathcal{S}^{*}(G/H)^{B,\chi\delta_{B}^{-1/2}})\leq\sum_{i=1}^{m}\sum_{k=0}^{\infty}\mathrm{dim}_{\mathbb{C}}((\mathrm{Sym^{k}}(N_{B(n_{j_{i}}),n_{j_{i}}}^{Q'})\otimes_{\mathbb{R}}\mathbb{C})^{B^{\theta_{n_{j_{i}}}},\chi\delta_{B^{\theta_{n_{j_{i}}}}}^{\text{\textminus}1}\delta_{B}^{1/2}}).
\]
 Notice that $B^{\theta_{n_{j_{i}}}}=\{b\in B|\,\theta_{n_{j_{i}}}(b)=n_{j_{i}}\theta(b)n_{j_{i}}^{-1}=b\}=B_{n_{j_{i}}}$.
Hence, there exists $n\in N_{G}(T)\cap Q$ such that $\mathrm{dim}_{\mathbb{C}}((\mathrm{Sym^{k}}(N_{B(n),n}^{Q'})\otimes_{\mathbb{R}}\mathbb{C})^{B^{\theta_{n}},\chi\delta_{B^{\theta_{n}}}^{\text{\textminus}1}\delta_{B}^{1/2}})>0$.
In particular,
\[
\mathrm{dim}_{\mathbb{C}}((\mathrm{Sym^{k}}(N_{B(n),n}^{Q'})\otimes_{\mathbb{R}}\mathbb{C})^{T^{\theta_{n}},\chi\delta_{B^{\theta_{n}}}^{\text{\textminus}1}\delta_{B}^{1/2}})>0.
\]
As $\theta$ satisfies property \ref{eq:star}, we have that $\delta_{B}^{1/2}|_{B^{\theta_{n}}}=\delta_{B^{\theta_{n}}}$
and we get the desired result. 
\end{proof}
We can now combine Step 1 and Step 2 as follows: assume that $\pi\in\mathcal{SAF}_{\mathrm{Irr}}(G)$
is $H$-distinguished. Step 1 showed that for any Borel $B$, $\mathcal{S}^{*}(G/H)^{B,\chi\cdot\delta_{B}^{-1/2}}\neq0$
where $\chi$ is dominant with respect to the choice of $B$. By Lemma
\ref{lemma 5.1} we can find $A\subseteq T\subseteq B$ such that
$T$ and $A$ are $\theta$-stable. Now we can use Step 2 and deduce
the existence of $n\in N_{G}(T)\cap Q$ such that:
\[
\mathrm{dim}\left(\mathrm{Sym}(N_{B(n),n}^{Q'})\otimes_{\mathbb{R}}\mathbb{C})^{T^{\theta_{n}},\chi}\right)>0.
\]

\section{Proof of the main theorem- Theorem 1.3}

In this section we assume the same setting as in Section 5 and prove
Theorem \ref{thm 1.3}. Let $G$ be a complex connected reductive
group, and $\mathfrak{g}$ be its Lie algebra, $K$ be a maximal compact
subgroup of $G$ with Lie algebra $\mathfrak{k}$, $\tau$ the Cartan
involution fixing $K$, and $\mathfrak{p}$ the $-1$ eigenspace of
$\tau$. Let $\mathfrak{a}_{\mathfrak{p}}$ be a maximal abelian subspace
of $\mathfrak{p}$, with the corresponding analytic subgroup $A_{\mathfrak{p}}$. 
\begin{prop}
\label{prop6.1:We-can-choose}We can choose $A_{\mathfrak{p}}\subseteq A\subseteq T\subseteq B$,
where $A_{\mathfrak{p}},A$ and $T$ are $\theta$-stable. \end{prop}
\begin{proof}
By Lemma \ref{lemma 5.1} we can choose $A\subseteq T\subseteq B$
such that $T$ and $A$ are $\theta$-stable. As $G$ is a complex
group, we can choose some Cartan involution $\tau$ such that $\mathfrak{a}=\mathrm{Lie}(A)$
is a maximal abelian subalgebra of $\mathfrak{p}$, i.e $\mathfrak{a}=\mathfrak{a}_{\mathfrak{p}}$
and $A_{\mathfrak{p}}=\mathrm{exp}(\mathfrak{a}_{\mathfrak{p}})\simeq\left(\mathbb{R}_{>0}\right)^{n}$
is the connected component of $\{e\}$ in $A\simeq\left(\mathbb{R}^{\times}\right)^{n}$.
As $\theta$ is a an algebraic morphism, $\theta(A_{\mathfrak{p}})$
is connected and contains $\{e\}$ thus $A_{\mathfrak{p}}$ is $\theta$-stable
as well.
\end{proof}
Using the above proposition, we fix such $A_{\mathfrak{p}}\subseteq A\subseteq T\subseteq B$.
Let $\pi\in\mathcal{SAF}_{\mathrm{Irr}}(G)$ and let $\chi$ be its
Langlands parameter such that $\pi$ is the unique quotient of $\mathrm{Ind}_{B}^{G}(\chi)$.
Recall that $\theta_{n}(g)=n\theta(g)n^{-1}$ and $Q'=\{g\in G|\theta(g)=g^{\text{\textminus}1}\}$.
Using Step 1 and Step 2, the next theorem clearly implies Theorem
1.3a):
\begin{thm}
\label{thm6.2} If there exists $n\in N_{G}(T)\cap Q'$ such that
$\mathrm{dim}_{\mathbb{C}}\left(\mathrm{Sym}(N_{B(n),n}^{Q'})\otimes_{\mathbb{R}}\mathbb{C})^{T^{\theta_{n}},\chi}\right)>0$,
then: 
\[
\mathrm{dim}_{\mathbb{C}}\left(\mathrm{Sym}(N_{B(n),n}^{Q'})\otimes_{\mathbb{R}}\mathbb{C})^{T^{\theta_{n}},\chi}\right)=1
\]
 and $\tilde{\pi}\simeq\theta\circ\pi$.\end{thm}
\begin{cor}
Theorem \ref{thm6.2} also implies Theorem 1.3b).\end{cor}
\begin{proof}
By Corollary \ref{Corollary 3.7} and the discussion in Section 5
\begin{eqnarray*}
\mathrm{dim}_{\mathbb{C}}\left(\pi^{*}\right)^{H} & \leq & \mathrm{dim}_{\mathbb{C}}\mathcal{S}^{*}(G/H)^{B,\chi\cdot\delta_{B}^{-1/2}}\\
 & \leq & \sum_{i=1}^{\left|B\backslash G/H\right|}\sum_{k=0}^{\infty}\mathrm{dim}_{\mathbb{C}}((\mathrm{Sym^{k}}(N_{B(n_{j_{i}}),n_{j_{i}}}^{Q'})\otimes_{\mathbb{R}}\mathbb{C})^{T^{\theta_{n_{i}}},\chi})
\end{eqnarray*}
where $n_{j_{i}}$ is a representative of the $B$-orbit $B(n_{j_{i}})$
in $Q$. By Theorem \ref{thm6.2} it follows that $\mathrm{dim}_{\mathbb{C}}\left(\pi^{*}\right)^{H}\leq\left|B\backslash G/H\right|$
and also that $\tilde{\pi}\simeq\theta\circ\pi$ as required.
\end{proof}
The proof of Theorem \ref{thm6.2} is divided to several parts: in
the first part we calculate $N_{B(n),n}^{Q'}$ using the Lie algebra
of $G$ by some direct calculations. In the second part we decompose
$\mathrm{Sym}(N_{B(n),n}^{Q'})\otimes_{R}\mathbb{C}$ to eigenspaces
as a complex representation of $T^{\theta_{n}}$ and then use the
$(T^{\theta_{n}},\chi)$ equivariance to show that $\chi^{-1}=w\circ\theta(\chi)$,
where $w=nT\in N_{G}(T)/T$ is the element in the Weyl group $W$
that corresponds to $n$. From here we deduce that $\tilde{\pi}\simeq\theta\circ\pi$.

\subsection{Calculation of $N_{B(n),n}^{Q'}$}

Recall the notations from Section \ref{sec:On the geometry of}. Define
a map $\gamma:G\longrightarrow G$ by $\gamma(g)=g\theta(g)$ and
let $a:B\longrightarrow B(n)$ be the orbit map $b\longmapsto bn\theta(b^{-1})$.
Notice that $T_{n}Q'=\mathrm{Ker}\left(d\gamma\right)_{n}$, $T_{n}B(n)=\mathrm{Im}\left(da\right)_{e}$
and $N_{B(n),n}^{Q'}=T_{n}Q'/T_{n}B(n)=\mathrm{Ker}\left(d\gamma\right)_{n}/\mathrm{Im}\left(da\right)_{e}$.
Denote by $\mu:G\times G\longrightarrow G$ the group multiplication,
by $L_{n},R_{n}:G\longrightarrow G$ the left and right multiplication
by $n$, and by $\left(dL_{n}\right)_{k}$ and $\left(dR_{n}\right)_{k}$,
their differentials at a point $k\in G$. For the convenience of the
reader, we will drop the notation $\left(\,\right)_{e}$ when taking
differential at $e$. 
\begin{lem}
\label{Lemma 6.4}We have that: 
\begin{eqnarray*}
T_{n}Q' & = & \mathrm{Ker}\left(d\gamma\right)_{n}=\{X\in T_{n}(G)|\left(dR_{n^{-1}}\right)_{n}(X)=-d\theta\circ\left(dL_{n^{-1}}\right)_{n}(X)\}.\\
T_{n}B(n) & = & \mathrm{Im}(da)=\{dR_{n}(Y)-dL_{n}\circ d\theta(Y))|\,Y\in T_{e}(B)\}.
\end{eqnarray*}
\end{lem}
\begin{proof}
$\gamma=\mu\circ(\mathrm{Id},\theta)$. We can write $\gamma=\mu\circ(R_{n^{-1}}\circ\mathrm{Id},L_{n}\circ\theta)$
and get:
\begin{eqnarray*}
\left(d\gamma\right)_{n}(X) & = & d\mu_{(e,e)}\circ(\left(dR_{n^{-1}}\right)_{n}\circ\left(d\mathrm{Id}\right)_{n},\left(dL_{n}\right)_{n^{-1}}\circ\left(d\theta\right)_{n})(X)\\
 & = & \left(dR_{n^{-1}}\right)_{n}(X)+\left(dL_{n}\right)_{n^{-1}}\circ\left(d\theta\right)_{n}(X).
\end{eqnarray*}
As $n\in Q'$, it satisfies $\theta(n)=n^{-1}$ so $n^{-1}\theta(n^{-1}g)=\theta(g)$
and this implies that 
\[
\left(d\theta\right)_{n}(X)=dL_{n^{-1}}\circ d\theta\circ\left(dL_{n^{-1}}\right)_{n}(X).
\]
 Altogether:
\begin{eqnarray*}
\left(d\gamma\right)_{n}(X) & = & \left(dR_{n^{-1}}\right)_{n}(X)+\left(dL_{n}\right)_{n^{-1}}\circ\left(d\theta\right)_{n}(X)\\
 & = & \left(dR_{n^{-1}}\right)_{n}(X)+\left(dL_{n}\right)_{n^{-1}}\circ dL_{n^{-1}}\circ d\theta\circ\left(dL_{n^{-1}}\right)_{n}(X)\\
 & = & \left(dR_{n^{-1}}\right)_{n}(X)+d\theta\circ\left(dL_{n^{-1}}\right)_{n}(X).
\end{eqnarray*}

This provides a description of 
\[
\mathrm{Ker}\left(d\gamma\right)_{n}=\{X\in T_{n}(G)|\left(dR_{n^{-1}}\right)_{n}(X)=-d\theta\circ\left(dL_{n^{-1}}\right)_{n}(X)\}.
\]
Now we would like to describe $T_{n}(B(n))=\mathrm{Im}(da)$. Notice
that $a(g)=gn\theta(g^{-1})=n\cdot(n^{-1}gn)\cdot\theta(g^{-1})$,
so:
\[
da(X)=dL_{n}\circ d\mu\circ(\mathrm{Ad_{n^{-1}}},d\left(\theta\circ i\right))(X)=dR_{n}(X)-dL_{n}\circ d\theta(X),
\]
 where $i:G\longrightarrow G$ is the inversion map. This finishes
the Lemma. 
\end{proof}
The goal now is to give a description of $T_{n}Q'$ and $T_{n}B(n)$
using the root space decomposition of $G$ as a reductive group. If
we consider $G$ as a complex group then we can write $\mathfrak{g}=\mathfrak{g}_{0}\oplus\bigoplus_{\widetilde{\alpha}\in\Sigma(G,T)}\mathfrak{g}_{\widetilde{\alpha}}$,
where $\mathfrak{g}_{0}=\mathrm{Lie}(T)$, $\Sigma(G,T)$ is the corresponding
root system and each eigenspace $\mathfrak{g}_{\widetilde{\alpha}}$
is one dimensional over $\mathbb{C}$. If we consider $G$ as a real
group then we can again write $\mathfrak{g}=\mathfrak{g}_{0}\oplus\bigoplus_{\alpha\in\Sigma(G,A_{\mathfrak{p}})}\mathfrak{g}_{\alpha}$.
As $A_{\mathfrak{p}}$ is $\theta$-stable we have an action of $\theta$
on the restricted root system $\Sigma(G,A_{\mathfrak{p}})$. We write
$\Sigma^{+}(G,A_{\mathfrak{p}})$, $\Sigma^{-}(G,A_{\mathfrak{p}})$
for the positive and negative roots (with respect to the choice of
the Borel subgroup $B$). 
\begin{lem}
\label{lemma 6.5-each eigen}Each eigenspace $\mathfrak{g}_{\alpha}$,
for $\alpha\in\Sigma(G,A_{\mathfrak{p}})$, is two dimensional over
$\mathbb{R}$ and $\mathfrak{g}_{\alpha}=\mathfrak{g}_{\widetilde{\alpha}}$
for some $\widetilde{\alpha}\in\Sigma(G,T)$ such that $\widetilde{\alpha}|_{A_{\mathfrak{p}}}=\alpha$.
Moreover, there is a unique such $\widetilde{\alpha}$ and $d\theta(\mathfrak{g}_{\widetilde{\alpha}})=\mathfrak{g}_{\widetilde{\theta(\alpha)}}$.\end{lem}
\begin{proof}
We have an isomorphism $\phi:T\simeq\left(\mathbb{C}^{\times}\right)^{n}$
where $\phi|_{A}:A\stackrel{\simeq}{\longrightarrow}(\mathbb{R}^{\times})^{n}$
and $\phi|_{A_{\mathfrak{p}}}:A_{\mathfrak{p}}\stackrel{\simeq}{\longrightarrow}\left(\mathbb{R}_{>0}\right)^{n}$.
Under this identification, we can write any root $\widetilde{\alpha}\in\Sigma(G,T)$
as $\widetilde{\alpha}(t)=\widetilde{\alpha}(t_{1},...,t_{n})=\prod\left|t_{i}\right|^{n_{i}}\cdot\left(\frac{t_{i}}{\left|t_{i}\right|}\right)^{n_{i}},$
where $n_{i}\in\mathbb{Z}$. If we restrict $\widetilde{\alpha}$
to $A_{\mathfrak{p}}$ we can write $\widetilde{\alpha}|_{A_{\mathfrak{p}}}(t_{1},...,t_{n})=\prod t_{i}^{n_{i}}$.
Therefore $\alpha=\widetilde{\alpha}|_{A_{\mathfrak{p}}}=\widetilde{\beta}|_{A_{\mathfrak{p}}}=\beta$
if and only if $\widetilde{\alpha}=\widetilde{\beta}$. This implies
that $d\theta(\mathfrak{g}_{\widetilde{\alpha}})=\mathfrak{g}_{\widetilde{\theta(\alpha)}}$.
\end{proof}
We now calculate $T_{n}Q'$ and $T_{n}B(n)$ in terms of the restricted
root decomposition $\mathfrak{g}=\mathfrak{g}_{0}\oplus\bigoplus_{\alpha\in\Sigma(G,A_{\mathfrak{p}})}\mathfrak{g}_{\alpha}$.
Denote $\psi_{n}:=d\theta\circ\mathrm{Ad}_{n}$, $\mathfrak{b}=\mathfrak{g}_{0}\oplus\bigoplus_{\alpha\in\Sigma^{+}(G,A_{\mathfrak{p}})}\mathfrak{g}_{\alpha}$.
Let $w\in W(G,A_{\mathfrak{p}})$ be the Weyl element corresponding
to $n\in N_{G}(T)=N_{G}(A_{\mathfrak{p}})$.
\begin{thm}
Under the identification of $T_{n}Q'$ and $T_{n}B(n)$ as sub-algebras
of $\mathfrak{g}$ by applying $dL_{n^{-1}}$, we have: 
\begin{eqnarray*}
T_{n}(Q') & \simeq & \{X\in\mathfrak{g}|d\theta\circ\mathrm{Ad}_{n}(X)=-X\}=\left(\mathfrak{g}_{0}\oplus\bigoplus_{\alpha\in\Sigma}\mathfrak{g}_{\alpha}\right)_{\psi_{n},-1}=\left(\mathfrak{g}_{0}\oplus\bigoplus_{\alpha\in\Sigma}\mathfrak{g}_{w^{-1}(\alpha)}\right)_{\psi_{n},-1},\\
T_{n}B(n) & \simeq & \{\mathrm{Ad}_{n^{-1}}(Y)-d\theta(Y)|Y\in\mathfrak{b}\}=\left(\mathfrak{g}_{0}\oplus\bigoplus_{\alpha\text{or }w\circ\theta(\alpha)\in\Sigma^{+}}\mathfrak{g}_{w^{-1}(\alpha)}\right)_{\psi_{n},-1},
\end{eqnarray*}
where $\left(\,\right)_{\psi_{n},-1}$ denote the $-1$ eigenspace
of $\psi_{n}$.\end{thm}
\begin{proof}
The first statement is clear by using Lemma \ref{Lemma 6.4} and identify
$T_{n}Q'$ with $T_{e}(G)$ by applying $\left(dL_{n^{-1}}\right)_{n}$.
For the second statement, observe that $\{\mathrm{Ad}_{n^{-1}}(Y)-d\theta(Y)|Y\in\mathfrak{g}_{\alpha}\}$
is the $-1$ eigenspace of $\psi_{n}$ restricted to the space $\mathfrak{g}_{w^{-1}(\alpha)}\oplus\mathfrak{g}_{\theta(\alpha)}$
(or the space $\mathfrak{g}_{w^{-1}(\alpha)}$ if $\psi_{n}(\mathfrak{g}_{w^{-1}(\alpha)})=\mathfrak{g}_{w^{-1}(\alpha)}$).
Indeed, $\{\mathrm{Ad}_{n^{-1}}(Y)-d\theta(Y)|Y\in\mathfrak{g}_{\alpha}\}$
and $\{\mathrm{Ad}_{n^{-1}}(Y)+d\theta(Y)|Y\in\mathfrak{g}_{\alpha}\}$
are contained in the $-1$ and $+1$ eigenspaces of $\psi_{n}|_{\mathfrak{g}_{w^{-1}(\alpha)}\oplus\mathfrak{g}_{\theta(\alpha)}}$,
respectively, and since they span $\mathfrak{g}_{w^{-1}(\alpha)}\oplus\mathfrak{g}_{\theta(\alpha)}$
it follows that they are indeed the $\pm1$ eigenspaces. In addition,
notice that 
\[
\{\mathrm{Ad}_{n^{-1}}(Y)-d\theta(Y)|Y\in\mathfrak{g}_{\alpha}\}=\{\mathrm{Ad}_{n^{-1}}(Y)-d\theta(Y)|Y\in\mathfrak{g}_{w\circ\theta(\alpha)}\}
\]
 are both the $-1$ eigenspace of $\psi_{n}$ restricted to the space
$\mathfrak{g}_{w^{-1}(\alpha)}\oplus\mathfrak{g}_{\theta(\alpha)}$.
The same argument can be applied to $\mathfrak{g}_{0}$. This finishes
the second statement.\end{proof}
\begin{cor}
Under the identification of $T_{n}Q'$ and $T_{n}B(n)$ as sub-algebras
of $\mathfrak{g}$ by applying $dL_{n^{-1}}$, we obtain:

\[
N_{B(n),n}^{Q'}=T_{n}Q'/T_{n}B(n)\simeq\left(\bigoplus_{\alpha\in\Sigma^{-}\cap w\circ\theta(\Sigma^{-})}\mathfrak{g}_{w^{-1}(\alpha)}\right)_{\psi_{n},-1}.
\]

\end{cor}

\subsection{Proof of Theorem \ref{thm6.2}}

We are now interested in $\mathrm{Sym}(N_{B(n),n}^{Q'})\otimes_{\mathbb{R}}\mathbb{C}$
as a complex representation of $T^{\theta_{n}}$ under the twisted
involution $\rho_{t}(g):=tg\theta(t)^{-1}$. Notice that $T^{\theta_{n}}$
acts on $N_{B(n),n}^{Q'}$ via the differential of the action at the
point $n$, i.e via $\left(d\rho_{t}\right)_{n}$, and this induces
an action on $\mathrm{Sym}(N_{B(n),n}^{Q'})\otimes_{\mathbb{R}}\mathbb{C}$. 
\begin{lem}
\label{lemma 6.8} Consider $N_{B(n),n}^{Q'}\simeq\left(\bigoplus_{\alpha\in\Sigma^{-}\cap w\circ\theta(\Sigma^{-})}\mathfrak{g}_{w^{-1}(\alpha)}\right)_{\psi_{n},-1}$
as a representation of $T^{\theta_{n}}$ (with respect to $\left(d\rho_{t}\right)_{n}$).
Then each $\left(\mathfrak{g}_{w^{-1}(\alpha)}\oplus\mathfrak{g}_{\theta(\alpha)}\right)_{\psi_{n},-1}$
is an $\widetilde{\alpha}$-eigenspace, for $\widetilde{\alpha}\in\Sigma(G,T)$
(i.e, the unique root $\widetilde{\alpha}$ such that $\widetilde{\alpha}|_{A_{\mathfrak{p}}}=\alpha$,
see Lemma \ref{lemma 6.5-each eigen}).\end{lem}
\begin{proof}
Observe that $\rho_{t}(g)=tg\theta(t)^{-1}=tgt^{-1}\cdot t\theta(t)^{-1}$
and that $t^{-1}n\theta(t)n^{-1}=e\Longleftrightarrow n^{-1}tn=\theta(t)$
for any $t\in T^{\theta_{n}}$. In other words, $\mathrm{Ad}_{\theta(t)}=\mathrm{Ad}_{n^{-1}tn}$.
$T^{\theta_{n}}$ acts on $T_{n}(Q')$ via $\left(d\rho_{t}\right)_{n}$
and recall that we have identified $T_{n}(Q')$ with a subalgebra
of $\mathfrak{g}$ via $\left(dL_{n^{-1}}\right)_{n}$ $(6.1)$. Therefore,
in order to understand the action of $T^{\theta_{n}}$ on $\left(\mathfrak{g}_{w^{-1}(\alpha)}\oplus\mathfrak{g}_{\theta(\alpha)}\right)_{\psi_{n},-1}$
we need to identify it with $T_{n}(Q')$ via $dL_{n}$, apply the
twisted involution and move back to $\mathfrak{g}$ via $\left(dL_{n^{-1}}\right)_{n}$.
Define $\rho'_{t}:=L_{n^{-1}}\circ\rho_{t}\circ L_{n}$ to be induced
twisted action on $T_{e}(G)\simeq\mathfrak{g}$. Notice that: 
\[
\rho'_{t}(g)=n^{-1}\rho_{t}(ng)=n^{-1}tng\theta(t)^{-1}=n^{-1}tngn^{-1}t^{-1}n.
\]
Therefore: 
\[
d\rho'_{t}(X):=dL_{n^{-1}}\circ d\rho_{t}\circ dL_{n}(X)=\mathrm{Ad}_{n^{-1}}\circ\mathrm{Ad}_{t}\circ\mathrm{Ad}_{n}(X)=\mathrm{Ad}_{n^{-1}tn}(X)=\mathrm{Ad}_{\theta(t)}(X),
\]
and as a consequence:
\[
d\rho_{t}'(\mathfrak{g}_{w^{-1}(\alpha)})=w^{-1}(\widetilde{\alpha})(\theta(t))\mathfrak{g}_{w^{-1}(\alpha)}=\theta\circ w^{-1}\circ\widetilde{\alpha}(t)\mathfrak{g}_{w^{-1}(\alpha)}=\widetilde{\alpha}(t)\mathfrak{g}_{w^{-1}(\alpha)}
\]
and also $d\rho'_{t}(\mathfrak{g}_{\theta(\alpha)})=\widetilde{\alpha}(t)\mathfrak{g}_{\theta(\alpha)}$.
This implies that $\left(\mathfrak{g}_{w^{-1}(\alpha)}\oplus\mathfrak{g}_{\theta(\alpha)}\right)_{\psi_{n},-1}$
is an eigenspace of $T^{\theta_{n}}$ under twisted involution with
the same character $\widetilde{\alpha}$. \end{proof}
\begin{cor}
\textup{\label{cor6.9}}As a complex representation of \textup{$T^{\theta_{n}}$,
we have that:}

\[
N_{B(n),n}^{Q'}\otimes_{R}\mathbb{C}\simeq\bigoplus_{\alpha\in S}V_{\widetilde{\alpha}},
\]

where $S$ is some multiset containing only negative roots $\alpha\in\Sigma(G,A_{\mathfrak{p}})$
such that $w\circ\theta(\alpha)$ is also negative, and $V_{\widetilde{\alpha}}$
is a one dimensional representation of $T^{\theta_{n}}$ corresponding
to a character $\widetilde{\alpha}$.\end{cor}
\begin{proof}
The space $N_{B(n),n}^{Q'}\otimes_{\mathbb{R}}\mathbb{C}$ is a complex
representation of $T^{\theta_{n}}$ so we can decompose it to one-dimensional
$\mathbb{C}$-eigenspaces. By Lemma \ref{lemma 6.8}, $\left(\mathfrak{g}_{w^{-1}(\alpha)}\oplus\mathfrak{g}_{\theta(\alpha)}\right)_{\psi_{n},-1}$
is an $\widetilde{\alpha}$-eigenspace, thus:

{*} If $w\circ\theta(\alpha)\neq\alpha$ we have that $\mathbb{C}\otimes_{\mathbb{R}}\left(\mathfrak{g}_{w^{-1}(\alpha)}\oplus\mathfrak{g}_{\theta(\alpha)}\right)_{\psi_{n},-1}\simeq V_{\widetilde{\alpha}}\oplus V_{\widetilde{\alpha}}$. 

{*} If $w\circ\theta(\alpha)=\alpha$ and $\psi_{n}|_{\mathfrak{g}_{\theta(\alpha)}}=-\mathrm{Id}|_{\mathfrak{g}_{\theta(\alpha)}}$
we have that $\mathbb{C}\otimes_{\mathbb{R}}\left(\mathfrak{g}_{\theta(\alpha)}\right)_{\psi_{n},-1}\simeq\mathfrak{g}_{\theta(\alpha)}\oplus\mathfrak{g}_{\theta(\alpha)}\simeq V_{\widetilde{\alpha}}\oplus V_{\widetilde{\alpha}}$. 

{*} If $w\circ\theta(\alpha)=\alpha$ and $\psi_{n}|_{\mathfrak{g}_{\theta(\alpha)}}=\mathrm{Id}|_{\mathfrak{g}_{\theta(\alpha)}}$
then $\mathbb{C}\otimes_{\mathbb{R}}\left(\mathfrak{g}_{\theta(\alpha)}\right)_{\psi_{n},-1}$
is trivial. 

This implies that $N_{B(n),n}^{Q'}\otimes_{\mathbb{R}}\mathbb{C}\simeq\bigoplus_{\alpha\in S}V_{\widetilde{\alpha}}$
when $S$ is a multiset that contains roots $\alpha$ (with possible
repetitions), such that both $\alpha$ and $w\circ\theta(\alpha)$
are negative. \end{proof}
\begin{example*}
Consider the case of $G=\mathrm{GL_{n}}(\mathbb{C}),\,G^{\theta}=\mathrm{GL_{n}}(\mathbb{R})$.
In this case the $B$-orbits in $Q'$ are represented by the set of
involutive Weyl elements, that is, $w\in W(G,A_{\mathfrak{p}})$ such
that $w^{2}=\mathrm{Id}$. In addition $W(G,A_{\mathfrak{p}})=S_{n}$-
the permutation group of $n$ elements. Under this identification,
we have an action of $W\simeq S_{n}$ on rows and columns. We also
have a decomposition $\mathfrak{gl}_{n}=\mathfrak{g}_{0}\oplus\bigoplus\mathfrak{g}_{\alpha_{ij}}$
where $\mathfrak{g}_{\alpha_{ij}}$ is the complex space spanned by
matrix $E_{ij}$ (that has $1$ in the $ij$th entry and $0$ in all
the other entries). Calculating $N_{B(w),w}^{Q'}$ yields: 
\[
N_{B(w),w}^{Q'}\otimes_{R}\mathbb{C}\simeq\bigoplus_{\{(i,j)\in I_{w}\}}V_{\widetilde{\alpha_{ij}}},
\]
where $I_{w}=\{(i,j):i>j,w(i)>w(j)\}$. Under the choice of the upper
triangular matrices as the Borel subgroup, $I_{w}$ corresponds to
the set of negative roots $\{\alpha_{ij}\}$ such that $\alpha_{w(i),w(j)}=w(\alpha_{ij})=w\circ\theta(\alpha_{ij})$
is negative. We see that this example agrees with the result in \ref{cor6.9}. \end{example*}
\begin{cor}
We have the following isomorphism of representations of $T^{\theta_{n}}$:
\[
\mathrm{Sym}(N_{B(n),n}^{Q'})\otimes_{R}\mathbb{C}\simeq\bigoplus_{\varphi:S\longrightarrow\mathbb{Z}_{\geq0}}V_{\alpha^{\varphi}},
\]
where $\varphi:S\longrightarrow\mathbb{Z}_{\geq0}$ is some function,
and $V_{\alpha^{\varphi}}$ is a one dimensional representation of
$T^{\theta_{n}}$ that corresponds to a character $\alpha^{\varphi}(t)=\prod_{\alpha_{j}\in S}\left(\widetilde{\alpha_{j}}(t)\right)^{\varphi(\alpha_{j})}$
for the multiset $S$ as defined in Corollary \ref{cor6.9}. \end{cor}
\begin{proof}
Choose a basis $\{Y_{1},...,Y_{m}\}$ of $N_{B(n),n}^{Q'}\otimes_{\mathbb{R}}\mathbb{C}$
where each $Y_{i}$ is in some $V_{\widetilde{\alpha_{i}}}$. Recall
that the action of $T^{\theta_{n}}$ on $N_{B(n),n}^{Q'}$ is via
$\mathrm{Ad}_{\theta(t)}$. We can now describe the action of $T^{\theta_{n}}$
on the space 
\[
\mathrm{Sym}(N_{B(n),n}^{Q'})\otimes_{R}\mathbb{C}\simeq\mathrm{Sym}(N_{B(n),n}^{Q'}\otimes_{R}\mathbb{C})=\bigoplus_{k=0}^{\infty}\mathrm{Sym^{k}}(N_{B(n),n}^{Q'}\otimes_{R}\mathbb{C}).
\]
$\mathrm{Sym^{k}}(N_{B(n),n}^{Q'}\otimes_{R}\mathbb{C})$ has basis
of tensors $\bigotimes_{j=1}^{k}Y_{i_{j}}$(modulo the symmetrization
relations). Then for $t\in T^{\theta_{n}}$: 
\[
t.(\bigotimes_{j=1}^{k}Y_{i_{j}})=\bigotimes_{j=1}^{k}\mathrm{Ad}_{\theta(t)}(Y_{i_{j}})=\widetilde{\alpha_{i_{1}}}(t)\cdot...\cdot\widetilde{\alpha_{i_{k}}}(t)\cdot\bigotimes_{j=1}^{k}Y_{i_{j}}=\alpha^{\varphi}(t)\cdot\bigotimes_{j=1}^{k}Y_{i_{j}},
\]
 where $\varphi(\alpha_{j})$ is the number of appearances of $j$
in $\{i_{1},...,i_{k}\}$.\end{proof}
\begin{cor}
\label{cor 6.11} If $\mathrm{dim}_{\mathbb{C}}\left(\mathrm{Sym}(N_{B(n),n}^{Q'}\otimes_{R}\mathbb{C})\right)^{T^{\theta_{n}},\chi}\neq0$
then there exists $\varphi:S\longrightarrow\mathbb{Z}_{\geq0}$ such
that $\alpha^{\varphi}(t)=\chi(t)$ for any $t\in T^{\theta_{n}}$.\end{cor}
\begin{proof}
Each vector $X\in\left(\mathrm{Sym}(N_{B(n),n}^{Q'}\otimes_{R}\mathbb{C})\right)^{T^{\theta_{n}},\chi}=\left(\bigoplus_{\varphi:S\longrightarrow\mathbb{Z}_{\geq0}}V_{\alpha^{\varphi}}\right)^{T^{\theta_{n}},\chi}$
is of the form $X=\sum_{i=1}^{n}X_{\varphi_{i}}$ where $X_{\varphi_{i}}\in V_{\alpha^{\varphi_{i}}}$.
This implies that for any $i$, $X_{\varphi_{i}}$ is also $(T^{\theta_{n}},\chi)$
equivariant. This implies that $\varphi_{i}(t)=\chi(t)$ for any $t\in T^{\theta_{n}}$. \end{proof}
\begin{lem}
\label{lemma 6.12}For any $n\in N_{G}(T)\cap Q'$, we have that $\{n\theta(t)n^{-1}t|t\in T\}\subseteq T{}^{\theta_{n}}$. \end{lem}
\begin{proof}
Recall that $T{}^{\theta_{n}}=\{t\in T|tn\theta(t)^{-1}=n\}=\{t\in T|\theta(t)=n^{-1}tn\}$
and notice that:
\[
\theta\left(n\theta(t)n^{-1}t\right)=n^{-1}tn\theta(t)=\theta(t)n^{-1}tn=n^{-1}\left(n\theta(t)n^{-1}t\right)n.
\]
\end{proof}
\begin{cor}
\label{cor 6.13}If $\left(\mathrm{Sym}(N_{B(n),n}^{Q'}\otimes_{R}\mathbb{C})\right)^{T^{\theta_{n}},\chi}\neq0$
then there exists $\varphi:S\longrightarrow\mathbb{Z}_{\geq0}$ such
that:
\[
\alpha^{\varphi}(t)\cdot w\circ\theta(\alpha^{\varphi})(t)=\chi(t)\cdot w\circ\theta(\chi)(t),
\]
for any $t\in T$.\end{cor}
\begin{proof}
By Corollary \ref{cor 6.11} , $\alpha^{\varphi}|_{T^{\theta_{n}}}=\chi|_{T^{\theta_{n}}}$
and by Lemma \ref{lemma 6.12} we can deduce that for every $t\in T$,
$\alpha^{\varphi}(n\theta(t)n^{-1}t)=\chi(n\theta(t)n^{-1}t)$. Therefore:
\[
\alpha^{\varphi}(n\theta(t)n^{-1}t)=\alpha^{\varphi}(t)\cdot w\circ\theta(\alpha^{\varphi})(t)=\chi(t)\cdot w\circ\theta(\chi)(t)=\chi(n\theta(t)n^{-1}t).
\]

\end{proof}
By Corollary \ref{cor 3.6}, in order to prove that $\widetilde{\pi}\simeq\pi^{\theta}$
it is enough to show that $\chi^{-1}=w\circ\theta(\chi)$. Notice
that by Langlands' classification, we have that $\chi$ is dominant
while $\alpha^{\varphi}|_{A_{\mathfrak{p}}}$ and $w\circ\theta(\alpha^{\varphi})|_{A_{\mathfrak{p}}}$
are negative. This leads us to our main result.
\begin{thm}
If $\left(\mathrm{Sym}(N_{B(n),n}^{Q'}\otimes_{R}\mathbb{C})\right)^{T^{\theta_{n}},\chi}\neq0$,
then: 

1) $\mathrm{dim}_{\mathbb{C}}\left(\mathrm{Sym}(N_{B(n),n}^{Q'}\otimes_{R}\mathbb{C})\right)^{T^{\theta_{n}},\chi}=1$.

2) $\chi^{-1}=w\circ\theta(\chi)$ and hence $\widetilde{\pi}^{\theta}\simeq\pi$.\end{thm}
\begin{proof}
By Corollary \ref{cor 6.13}, $\alpha^{\varphi}(t)\cdot w\circ\theta(\alpha^{\varphi})(t)=\chi(t)\cdot w\circ\theta(\chi)(t)$
for any $t\in T$ . For any root $\alpha_{j}$ with multiple instances
$\alpha_{j_{1}},...,\alpha_{j_{m}}$ in $S$, we may define $n_{j}=\sum_{i=1}^{m}\varphi(\alpha_{j_{i}})$.
We can change $S$ to be a set by deleting all multiple instances
in it. By doing this, we can write $\alpha^{\varphi}(t)=\prod_{\alpha_{j}\in S}\left(\widetilde{\alpha_{j}}(t)\right)^{n_{j}}$,
when $S$ is now a set (without repetitions) of roots $\{\alpha_{j}\}$
such that both $\alpha_{j},w\circ\theta(\alpha_{j})$ are negative.
We now restrict to $A_{\mathfrak{p}}$ and take the differentials
of the characters to get the following equation:
\[
\sum_{j\in S}n_{j}\cdot(\alpha_{j}+w\circ\theta(\alpha_{j}))=d(\chi|_{A_{\mathfrak{p}}})+d(w\circ\theta(\chi)|_{A_{\mathfrak{p}}}).
\]
Recall that both $\theta$ and $w$ preserve $\Sigma(G,A_{\mathfrak{p}})$
and hence preserve the inner product between the roots (see \cite[Lemma, p.43]{Hum72}),
that is: 
\[
\langle w\circ\theta(\alpha),w\circ\theta(\beta)\rangle=\langle w(\alpha),w(\beta)\rangle=\langle\theta(\alpha),\theta(\beta)\rangle=\langle\alpha,\beta\rangle.
\]
Thus, we have for any $\alpha\in\Sigma(G,A_{\mathfrak{p}})$:
\[
\langle d(\chi|_{A_{\mathfrak{p}}}),w\circ\theta(\alpha)\rangle=\langle d(w\circ\theta(\chi)|_{A_{\mathfrak{p}}}),w\circ\theta\circ w\circ\theta(\alpha)\rangle=\langle d(w\circ\theta(\chi)|_{A_{\mathfrak{p}}}),\alpha\rangle.
\]

We now claim that $\sum_{j\in S}n_{j}\cdot(\alpha_{j}+w\circ\theta(\alpha_{j}))=0$.
Assume that there exists some $k$ such that $n_{k}\neq0$. By taking
the inner product of $\sum_{j\in S}n_{j}\cdot(\alpha_{j}+w\circ\theta(\alpha_{j}))$
with itself we get on one hand:
\[
\langle\sum_{j\in S}n_{j}\cdot(\alpha_{j}+w\circ\theta(\alpha_{j})),\sum_{j\in S}n_{j}\cdot(\alpha_{j}+w\circ\theta(\alpha_{j}))\rangle>0.
\]
But on the other hand, as $\chi$ is dominant and $\alpha_{j},w\circ\theta(\alpha_{j})\in\Sigma(G,A_{\mathfrak{p}})^{-}$
for any $\alpha_{j}\in S$ it holds that:
\[
\langle d(\chi|_{A_{\mathfrak{p}}})+d(w\circ\theta(\chi)|_{A_{\mathfrak{p}}}),\sum_{j\in S}n_{j}\cdot(\alpha_{j}+w\circ\theta(\alpha_{j}))\rangle
\]
\[
=\langle d(\chi|_{A_{\mathfrak{p}}}),\sum_{j\in S}n_{j}\cdot(\alpha_{j}+w\circ\theta(\alpha_{j}))\rangle+\langle d(w\circ\theta(\chi)|_{A_{\mathfrak{p}}}),\sum_{j\in S}n_{j}\cdot(\alpha_{j}+w\circ\theta(\alpha_{j}))\rangle\leq0.
\]
We got a contradiction, hence $n_{k}=0$ for any $k$ and $\alpha^{\varphi}(n\theta(t)n^{-1}t)=1$.
This implies that:

1) $\chi^{-1}=w\circ\theta(\chi)$. 

2) $\left(\mathrm{Sym}(N_{B(n),n}^{Q'}\otimes_{R}\mathbb{C})\right)^{T^{\theta_{n}},\chi}=\left(\mathrm{Sym^{0}}(N_{B(n),n}^{Q'}\otimes_{R}\mathbb{C})\right)^{T^{\theta_{n}},\chi}$
is one dimensional. 
\end{proof}
\appendix

\section{\label{sec:Langlands'-classification}Langlands' classification}

Let $G$ be a real reductive group with Lie algebra $\mathfrak{g}$
and $K$ a maximal compact subgroup. Define $\mathcal{HC}$ to be
the category of finitely generated, admissible $(\mathfrak{g},K)$-modules
(see \cite[Lecture 5]{KT00}). Casselman and Wallach showed that there
is an equivalence of categories $\mathcal{SAF}(G)\simeq\mathcal{HC}(\mathfrak{g},K)$.
The functor in one direction is $E\longmapsto E^{K-finite}$, where
$E^{K-finite}$ is the subspace of $K$-finite vectors. The functor
in the other direction is: $V\longmapsto V^{\infty}$ where $V^{\infty}:=\pi(\mathcal{S}(G))V$,
$\mathcal{S}(G)$ is the Schwartz algebra of rapidly decreasing functions
on $G$, and $\pi(\mathcal{S}(G))V$ stands for the vector space spanned
by $\pi(f)v$ for $f\in\mathcal{S}(G)$, $v\in V$.

The Langlands classification gives a classification of all equivalence
classes of irreducible admissible \textbf{Hilbert representations}
of $G$ up to infinitesimal equivalence, where two irreducible Hilbert
representations $(\pi_{1},\mathcal{H}_{1})$ and $(\pi_{2},\mathcal{H}_{2})$
are $\emph{infinitesimal equivalent}$ if $\mathcal{H}_{1}^{K-finite}\simeq\mathcal{H}_{2}^{K-finite}$
as $(\mathfrak{g},K)$-modules. 

In this section we state Langlands' classification for Hilbert representations
and use it to prove the complex version (i.e, when $G$ is a complex
reductive group, see Theorem \ref{thm3.1}) of the classification
for Fréchet representations. \textcolor{black}{Both theorems are known
results, an}d in fact, the complex version was first proved by Zhelobenko
and Naimark (see \cite{ZN66}), and later it was generalized by Langlands. 

For the following definitions and Theorem \ref{thm3.1}, assume that
$\pi$ is a\textbf{ }Hilbert representations of $G$. Moreover, to
avoid confusion, we set $\mathrm{H}-\mathrm{Ind}_{P}^{G}$ to be the
normalized induction in the category of Hilbert representations, and
by $\mathrm{Ind}_{P}^{G}$ the normalized induction in the category
$\mathcal{SAF}(G)$. 
\begin{defn*}
1) $\pi$ is called a $\emph{discrete series}$ representation if
it is unitary, irreducible and it is infinitesimally equivalent to
a direct summand of the right regular representation of $G$ on $L^{2}(G)$.

2) $\pi$ is called a $\emph{tempered}$ representation of $G$, if
$\pi$ is infinitesimally equivalent to a subrepresentation of an
induced representation $H-\mathrm{Ind}_{MAN}^{G}(\sigma\otimes\lambda\otimes1)$
for some standard parabolic subgroup $MAN\supseteq M_{p}A_{p}N_{p}$,
some discrete series representation $\sigma$ of $M$ , and some imaginary
character $\lambda$ of $A$, i.e of the form $(t_{1},...,t_{n})\longmapsto\prod t_{i}^{\lambda_{i}}$
where the $\lambda_{i}$ are purely imaginary. \end{defn*}
\begin{thm}
\label{thm:(The-Langlands'-Classification)}Let $G$ be a real connected
reductive Lie group, with $\mathfrak{g}$ be its Lie algebra. Also
let $K$ be a maximal compact subgroup of $G$ with Lie algebra $\mathfrak{k}$,
$\tau$ the Cartan involution fixing $K$, and $\mathfrak{p}$ the
$-1$ eigenspace of $\tau$. Let $\mathfrak{a}_{\mathfrak{p}}$ be
a maximal abelian subspace of $\mathfrak{p}$, $\Sigma$ the root
system of $\mathfrak{a}_{\mathfrak{p}}$ in $\mathfrak{g}$, $\triangle$
the set of simple roots of $\Sigma$, and $F\subseteq\triangle$.
Furthermore, let $M_{\mathfrak{p}}A_{\mathfrak{p}}N_{\mathfrak{p}}$
be the corresponding minimal parabolic subgroup (see Appendix \ref{subC1:Parabolic-subgroups}).

Then the infinitesimal equivalence class of irreducible admissible
Hilbert representations of $G$ are uniquely parameterized by triples
$(F,[\sigma],\lambda)$, where:

{*} $Q_{F}$ is a standard parabolic subgroup of $G$ corresponding
to $F$ with Langlands decomposition $Q_{F}=M_{F}A_{F}N_{F}$ containing
$M_{\mathfrak{p}}A_{\mathfrak{p}}N_{\mathfrak{p}}$.

{*} $\sigma$ is an irreducible tempered representation of the semisimple
Lie group $M$, where $[\,]$ denotes equivalence class up to infinitesimal
equivalence.

{*} The character $\lambda\in\mathfrak{a}_{F}^{*}$ satisfies $\mathrm{Re}\langle\alpha,\lambda\rangle>0$
for all simple roots $\alpha$ not in $F$.

More precisely, the triple $(F,[\sigma],\lambda)$ corresponds to
the equivalence class of the unique irreducible quotient of $\mathrm{H}-\mathrm{Ind}_{MAN}^{G}(\sigma\otimes\lambda\otimes1)$,
that is, $\pi\simeq\mathrm{H}-\mathrm{Ind}_{MAN}^{G}(\sigma\otimes\lambda\otimes1)/W$,
where $W$ is the maximal subrepresentation of $\mathrm{H}-\mathrm{Ind}_{MAN}^{G}(\sigma\otimes\lambda\otimes1)$. 
\end{thm}

\subsection{\label{subA1Proof-of-Theorem}Proof of Theorem \ref{thm3.1}}

Before we prove Theorem \ref{thm3.1} we present three facts and key
Theorem \ref{Theorem A.2}:
\begin{fact*}
1) (Harish-Chandra, 1966, see e.g. \cite[p.71]{KT00}) A linear connected
reductive group has no discrete series unless $G$ has a compact Cartan
subgroup. 

2) Let $G$ be a complex reductive group and $\chi$ be a unitary
character of $T=M_{\mathfrak{p}}A_{\mathfrak{p}}$. Then $\mathrm{H}-\mathrm{Ind}_{B}^{G}\chi$
is irreducible (follows from \cite[Corollary 14.60]{Kn01}).

3) (Double induction formula, \cite[VII.2-property 4]{Kn01}) Let
$G$ a reductive group, $MAN$ be a parabolic subgroup of G and let
$M'A'N'$ be a parabolic subgroup of $M$, so that $M'(A'A)(N'N)$
is a parabolic subgroup of $G$. If $\sigma'$ is a unitary representation
of $M'$ and $v'$ and $v$ are characters of $\mathfrak{a}'$ and
$\mathfrak{a}$, respectively, then there is a canonical equivalence:
\[
\mathrm{H}-\mathrm{Ind}_{MAN}^{G}(\mathrm{H}-\mathrm{Ind}_{M'A'N'}^{M}(\sigma'\otimes\mathrm{exp}(v')\otimes1)\otimes\mathrm{exp}(v)\otimes1\simeq\mathrm{H}-\mathrm{Ind}_{M'(A'A)(N'N)}^{G}(\sigma'\otimes\mathrm{exp}(v'+v)\otimes1)
\]
\end{fact*}
\begin{thm}
\label{Theorem A.2}Let $\pi$ be an irreducible, admissible, Hilbert
representation of $G$ and fix some maximal torus $T$ and a Borel
subgroup $B\supseteq T$. Then it is equivalent to the unique quotient
of $H-\mathrm{Ind}_{B}^{G}(\chi)$, where $\chi$ is some dominant
character of $T$, i.e $\mathrm{Re}\langle d\chi|_{\mathfrak{a}_{\mathfrak{p}}},\alpha\rangle\geq0$.\end{thm}
\begin{proof}
We choose $B=M_{\mathfrak{p}}A_{\mathfrak{p}}N_{\mathfrak{p}}$. By
the Langlands classification, $\pi$ is equivalent to $\mathrm{H}-\mathrm{Ind}_{MAN}^{G}(\sigma\otimes\lambda\otimes1)/W$
where $Q=MAN\supseteq B=M_{\mathfrak{p}}A_{\mathfrak{p}}N_{\mathfrak{p}}$,
and $W$ is the maximal proper subrepresentation. We have two possible
cases:

1) If $MAN=M_{\mathfrak{p}}A_{\mathfrak{p}}N_{\mathfrak{p}}=B$ then
$\pi$ is equivalent to the unique quotient of $\mathrm{H}-\mathrm{Ind}_{B}^{G}(\sigma\otimes\lambda\otimes1)$,
where $\lambda$ is a strongly dominant character (i.e $\mathrm{Re}\langle d\lambda,\alpha\rangle>0$)
and $\sigma$ is an unitary representation of $M_{\mathfrak{p}}$.
As $M_{\mathfrak{p}}$ is compact and abelian, $\sigma$ is a one
dimensional unitary representation. As a consequence, $\chi:=\sigma\cdot\lambda$
is a character and $\chi|_{A_{\mathfrak{p}}}=\lambda$ so it is dominant. 

2) $B\varsubsetneq Q=MAN\subseteq G$. Then $\pi$ is equivalent to
$\pi'=\mathrm{H}-\mathrm{Ind}_{MAN}^{G}(\sigma\otimes\lambda\otimes1)/W$,
where $\sigma$ is a tempered representation of $M$. By definition,
$\sigma$ is equivalent to $\tau$, where $\tau\hookrightarrow\mathrm{H}-\mathrm{Ind}_{M'A'N'}^{M}(\sigma'\otimes\lambda'\otimes1)$,
$\sigma'$ is a discrete series representation of $M'$ and $\lambda'$
is an imaginary character of $A'$. We now show that:

{*} $M'$ has no discrete series unless $M'A'AN'$ is a minimal parabolic
of $MA$. 

{*} $\mathrm{H}-\mathrm{Ind}_{M'A'N'}^{M}(\sigma'\otimes\lambda'\otimes1)$
is irreducible for $\sigma'$ unitary and $\lambda'$ unitary and
purely imaginary. 

$MA$ is a complex group with $M'AA'N$ a complex parabolic subgroup.
From the structure theory of parabolic subgroups of complex reductive
groups, it follows that $M'$ has no compact Cartan subgroup unless
$M'AA'N$ is a minimal parabolic subgroup of $MA$. For the second
statement, we have an intertwining map: $\tau\hookrightarrow\mathrm{H}-\mathrm{Ind}_{M'A'N'}^{M}(\sigma'\otimes\lambda'\otimes1)$.
As $A$ is central in $MA$, we can tensor $\tau$ with the trivial
representation on $A$ to get an irreducible representation $\tau\otimes1$
such that: 
\[
\tau\otimes1\hookrightarrow\left(\mathrm{H}-\mathrm{Ind}_{M'A'N'}^{M}(\sigma'\otimes\lambda'\otimes1)\right)\otimes1\hookrightarrow\mathrm{H}-\mathrm{Ind}_{M'A'AN'}^{MA}(\sigma'\otimes(\lambda'\otimes1)\otimes1).
\]
But as $MA$ is complex and $M'A'AN'$ is a minimal parabolic subgroup,
i.e a Borel subgroup, it follows by Fact 2 that $\mathrm{H}-\mathrm{Ind}_{M'A'AN'}^{MA}(\sigma'\otimes(\lambda'\otimes1)\otimes1)$
is irreducible and hence $\mathrm{H}-\mathrm{Ind}_{M'A'N'}^{M}(\sigma'\otimes\lambda'\otimes1)=\tau$
is irreducible. By (\cite[Proposition 7.82]{Kn96}) we can find a
minimal parabolic subgroup $M'A'N'$ of $M$ such that $M'A'AN'N=M_{\mathfrak{p}}A_{\mathfrak{p}}N_{\mathfrak{p}}=B$. 

since $\sigma$ is equivalent to $\mathrm{H}-\mathrm{Ind}_{M'A'N'}^{M}(\sigma'\otimes\lambda'\otimes1)$,
we can assume that $\pi'$ is the unique quotient: 
\[
\mathrm{H}-\mathrm{Ind}_{MAN}^{G}(\left(H-\mathrm{Ind}_{M'A'N'}^{M}(\sigma'\otimes\lambda'\otimes1)\right)\otimes\lambda\otimes1)\longrightarrow\pi'.
\]
Using the double induction formula (Fact 3) on the left side of the
equation gives:
\begin{eqnarray*}
\mathrm{H}-\mathrm{Ind}_{B}^{G}(\chi) & \simeq & \mathrm{H}-\mathrm{Ind}_{M'A'AN'N}^{G}(\sigma'\otimes\lambda'\cdot\lambda\otimes1)\\
 & \simeq & \mathrm{H}-\mathrm{Ind}_{MAN}^{G}(\left(\mathrm{H}-\mathrm{Ind}_{M'A'N'}^{M}(\sigma'\otimes\lambda'\otimes1)\right)\otimes\lambda\otimes1)\longrightarrow\pi',
\end{eqnarray*}
where $\sigma'\lambda'$ is a unitary character of $MA'$, $\lambda'$
is purely imaginary and $\chi=\sigma'\lambda'\lambda$. Notice that
$\lambda$ satisfies $\mathrm{Re}\langle\alpha,\lambda\rangle>0$
for all simple roots $\alpha$ not in $F$ and $\mathrm{Re}\langle\alpha,\lambda\rangle=0$
for $\alpha\in F$. Therefore $\chi|_{A_{p}}=\lambda'\lambda$ satisfies
$\mathrm{Re}\langle d\chi|_{\mathfrak{a}_{\mathfrak{p}}},\alpha\rangle\geq0$
for any $\alpha\in\triangle$ and $\chi$ is a dominant character. 
\end{proof}
We now show that Theorem \ref{Theorem A.2} can be applied to Fréchet
representations: 

By a theorem of Harish-Chandra, any irreducible $V\in\mathcal{HC}$
admits a Hilbert globalization $\mathcal{H}$, such that $\mathcal{H}$
is an irreducible, admissible Hilbert representation and $\mathcal{H}^{K-finite}\simeq V$
as a $(\mathfrak{g},K)$-module (see e.g. \cite[Section 5.1]{BK14}).
Denote $\mathcal{H}_{smooth}$ to be the smooth part of $\mathcal{H}$
and notice that $\mathcal{H}_{smooth}\in\mathcal{SAF}_{\mathrm{Irr}}(G)$
is a Fréchet globalization of $V$. Assume we are given a representation
$E\in\mathcal{SAF}_{\mathrm{Irr}}(G)$ such that $V:=E^{K-finite}$.
Then by the Casselman-Wallach globalization theorem we have that $\mathcal{H}_{smooth}\simeq E$. 

We can now prove Theorem \ref{thm3.1}:
\begin{thm}
Let $(\pi,E)\in\mathcal{SAF}_{\mathrm{Irr}}(G)$. Then $(\pi,E)$
is the unique irreducible quotient of $\mathrm{Ind}_{B}^{G}(\chi)$
for a dominant character $\chi$ of $T$.\end{thm}
\begin{proof}
Denote $V:=E^{K-finite}$ and let $\mathcal{H}$ its Hilbert globalization,
such that $E\simeq\mathcal{H}_{smooth}$. By Theorem \ref{thm:(The-Langlands'-Classification)},
$\mathcal{H}$ is equivalent to $\mathrm{H}-\mathrm{Ind}_{B}^{G}(\chi)/W$
, where $W$ is the maximal subrepresentation of $\mathrm{H}-\mathrm{Ind}_{B}^{G}(\chi)$.
This implies that $V\simeq\mathrm{H}-\mathrm{Ind}_{B}^{G}(\chi)^{K-finite}/W^{K-finite}$.
Note that the $\mathrm{Ind}_{B}^{G}(\chi)^{K-finite}=\mathrm{H}-\mathrm{Ind}_{B}^{G}(\chi)^{K-finite}$
so $\mathrm{Ind}_{B}^{G}(\chi)$ is a Fréchet globalization of $\mathrm{H}-\mathrm{Ind}_{B}^{G}(\chi)^{K-finite}$
and this implies that $(\pi,E)$ is the unique quotient of $\mathrm{Ind}_{B}^{G}(\chi)$
as required. 
\end{proof}

\subsection{\label{sub:Proof-of-Corollary 3.4}Proof of Corollary \ref{cor 3.4}}
\begin{cor}
Let $\pi_{1},\pi_{2}\in\mathcal{SAF}_{\mathrm{Irr}}(G)$ where $\pi_{1}$
is the unique quotient of $\mathrm{Ind}_{B}^{G}(\chi_{1})$ and $\pi_{2}$
is the unique quotient of $\mathrm{Ind}_{B}^{G}(\chi_{2})$. Then
$\pi_{1}\simeq\pi_{2}$ if and only if there exists $w\in W(G,A_{\mathfrak{p}})$
such that $w\circ\chi_{1}=\chi_{2}$. \end{cor}
\begin{proof}
By Corollary \ref{cor:Let--and be unitary} this is true for tempered
representations. By Theorem \ref{Theorem A.2}, $\pi_{1}$ is the
unique quotient of $\mathrm{Ind}_{P}^{G}(\sigma_{1}\otimes v_{1}\otimes1)$
and $\pi_{2}$ is the unique quotient of $\mathrm{Ind}_{P}^{G}(\sigma_{2}\otimes v_{2}\otimes1)$.
Moreover, we can write $\sigma_{1}=\mathrm{Ind}_{M'A'N'}^{M}(\sigma_{1}'\otimes v_{1}'\otimes1)$
and $\sigma_{2}=\mathrm{Ind}_{M'A'N'}^{M}(\sigma_{2}'\otimes v_{2}'\otimes1)$
where $M'A'AN'N=M_{\mathfrak{p}}A_{\mathfrak{p}}N_{\mathfrak{p}}=B$,
such that $\pi_{1}$ is the unique quotient of $\mathrm{Ind}_{B}^{G}(\sigma_{1}'\otimes v_{1}'\cdot v_{1}\otimes1)$
and $\pi_{2}$ is the unique quotient of $\mathrm{Ind}_{B}^{G}(\sigma_{2}'\otimes v_{2}'\cdot v_{2}\otimes1)$.
Let $\chi_{1}=\sigma_{1}'\otimes v_{1}'\cdot v_{1}$ and $\chi_{2}=\sigma_{2}'\otimes v_{2}'\cdot v_{2}$. 

Assume $\pi_{1}\simeq\pi_{2}$. By Theorem \ref{Theorem A.2}, $\sigma_{1}\simeq\sigma_{2}$
and $v_{1}=v_{2}$. By Corollary \ref{cor:Let--and be unitary}, $\sigma_{1}\simeq\sigma_{2}$
if and only if there exists $w'\in W(M,A')$ such that $w'\circ(\sigma_{1}'\otimes v_{1}')=(\sigma_{2}'\otimes v_{2}')$.
Let $n'\in N_{M}(A')$ be an element that correspond to $w'$. $M$
centralizes $A$ and therefore $n'$ centralizes $A$. Thus $n'\in N_{G}(A_{\mathfrak{p}})$
and it corresponds to an element $w\in W(G,A_{\mathfrak{p}})$ such
that $w$ centralize $A$ and acts as $w'$ on $A'$. Therefore:
\[
w(\sigma_{1}'\otimes v_{1}'\cdot v_{1}\otimes1)=w'\sigma_{1}'\otimes w'v_{1}'\cdot v_{1}\otimes1=\sigma_{2}'\otimes v_{2}'\cdot v_{2}\otimes1,
\]
 which implies that $w\circ\chi_{1}=\chi_{2}$. For the other direction,
assume $w\circ\chi_{1}=\chi_{2}$, that is $w\circ(\sigma_{1}'\otimes v_{1}'\cdot v_{1}\otimes1)=\sigma_{2}'\otimes v_{2}'\cdot v_{2}\otimes1$.
As both $v_{1}'\cdot v_{1}$ and $v_{2}'\cdot v_{2}$ are dominant
with respect to positive roots $\alpha$ such that $\alpha|_{A}$
is non zero, $w$ must act trivially on $A$. Therefore $w(\sigma_{1}'\otimes v_{1}'\cdot v_{1}\otimes1)=w'\sigma_{1}'\otimes w'v_{1}'\cdot v_{1}\otimes1=\sigma_{2}'\otimes v_{2}'\cdot v_{2}\otimes1$
so $v_{1}=v_{2}$ and by Corollary \ref{cor:Let--and be unitary}
$\sigma_{1}\simeq\sigma_{2}$. By Theorem \ref{Theorem A.2}, $\pi_{1}\simeq\pi_{2}$.
\end{proof}

\section{\label{Append:Some-facts-about}Some facts about the structure of
real reductive groups}

For the convenience of the reader, we summarize some of the structure
theory of real reductive groups. This appendix is based on \cite{KT00}
and \cite{Kn96}. 

Any real reductive algebraic group $G$ can be represented by a 4-tuple
$(G,K,\theta,B(\,,\,))$, where $K$ is a maximal compact subgroup,
$\theta$ is an involution of $\mathfrak{g}=\mathrm{Lie}(G)$, and
$B(\,,\,)$ is a non degenerate, $Ad(G)$-invariant, $\theta$-invariant,
bilinear form, such that: 
\begin{itemize}
\item $\mathfrak{g}$ is a reductive Lie algebra and it decomposes into
$\mathfrak{g}=\mathfrak{k}\oplus\mathfrak{p}$, where $\mathfrak{k}$
is the $+1$ eigenspace of $\theta$ (and the Lie algebra of $K$)
and $\mathfrak{p}$ is the $-1$ eigenspace of $\theta$.
\item $B(\,,\,)$ is negative definite on $\mathfrak{k}$ and positive definite
on $\mathfrak{p}$, and $\mathfrak{k}$ and $\mathfrak{p}$ are orthogonal
under $B$. 
\item We have a diffeomorphism $K\times\mathrm{exp}(\mathfrak{p})\longrightarrow G$.
\item Any automorphism $\mathrm{Ad}(g)$ is inner for $g\in G$. 
\end{itemize}
We can define an inner product $\langle\,,\,\rangle$ on $\mathfrak{g}$
by $\langle X,Y\rangle=-B(X,\theta(Y))$. 
\begin{fact*}
Relative to the inner product $\langle\,,\,\rangle$ on $\mathfrak{g}$,
it holds that $\left(\mathrm{ad}_{X}\right)^{*}=\mathrm{ad}_{X}$
for any $X\in\mathfrak{p}$.
\end{fact*}
Let $\mathfrak{a}_{\mathfrak{p}}$ be the maximal abelian subspace
of $\mathfrak{p}$. By the above fact, $\{\mathrm{ad}_{X}\}_{X\in\mathfrak{a}_{\mathfrak{p}}}$
is a commuting family of symmetric operators on $\mathfrak{g}$ and
hence it is simultaneously diagonalizable. This gives a decomposition
$\mathfrak{g}=\mathfrak{g}_{0}\oplus\bigoplus_{\lambda\in\Sigma}\mathfrak{g}_{\lambda}$
where $\mathfrak{g}_{\lambda}:=\{X\in g|(\mathrm{ad}_{H})X=\lambda(H)X\text{ for all }H\in\mathfrak{a}_{\mathfrak{p}}\}$,
$\Sigma$ is the set of all $\lambda\in\mathfrak{a}_{\mathfrak{p}}^{*}$
such that $\mathfrak{g}_{\lambda}$ is non zero. Such $\lambda$ is
called a $\emph{restricted root}$. This decomposition is called \textit{restricted
root space decomposition} and it has some nice properties:
\begin{itemize}
\item $[\mathfrak{g}_{\mu},\mathfrak{g}_{\lambda}]\subseteq\mathfrak{g}_{\mu+\lambda}$.
\item $\mathfrak{g}_{\mu}$ and $\mathfrak{g}_{\lambda}$ are orthogonal
with respect to $\langle\,,\,\rangle$ if $\mu\neq\lambda$.
\item $\mathfrak{g}_{0}=\mathfrak{m}_{\mathfrak{p}}\oplus\mathfrak{a}_{\mathfrak{p}}$
where $\mathfrak{m}_{\mathfrak{p}}=Z_{\mathfrak{k}}(\mathfrak{a}_{\mathfrak{p}})$
is the centralizer of $\mathfrak{a}_{\mathfrak{p}}$ in $\mathfrak{k}$.
This is an orthogonal sum.
\end{itemize}
We now introduce a lexicographic order on $\mathfrak{\mathfrak{a}_{\mathfrak{p}}}^{*}$.
We fix an ordered basis $\lambda_{1},...,\lambda_{l}$ of $\mathfrak{a}_{\mathfrak{p}}^{*}$
and define $\lambda=\sum_{i}c_{i}\lambda_{i}$ to be $\emph{positive}$
if $c_{1}=...c_{k}=0$ and $c_{k+1}>0$ for some $0\leq k<l$. We
set $\lambda>\mu$ if $\lambda-\mu$ is positive and define $\Sigma^{+}$
to be the set of positive roots in $\Sigma$ and $\mathfrak{n}_{\mathfrak{p}}=\bigoplus_{\lambda\in\Sigma^{+}}\mathfrak{g}_{\lambda}$.
We have:
\begin{itemize}
\item $\mathfrak{g}=\mathfrak{k}\oplus\mathfrak{\mathfrak{a}_{\mathfrak{p}}}\oplus\mathfrak{n}_{\mathfrak{p}}$,
$\mathfrak{a}_{\mathfrak{p}}$ is abelian, $\mathfrak{n}_{\mathfrak{p}}$
is nilpotent, $\mathfrak{a}_{\mathfrak{p}}\oplus\mathfrak{n}_{\mathfrak{p}}$
is solvable and $[\mathfrak{a}_{\mathfrak{p}}\oplus\mathfrak{n}_{\mathfrak{p}},\mathfrak{a}_{\mathfrak{p}}\oplus\mathfrak{n}_{\mathfrak{p}}]=\mathfrak{n}_{\mathfrak{p}}$.
\item There are analytic subgroups $A_{\mathfrak{p}}$ and $N_{\mathfrak{p}}$
with Lie algebras $\mathfrak{a}_{\mathfrak{p}},\mathfrak{n}_{\mathfrak{p}}$
such that $A_{\mathfrak{p}},N_{\mathfrak{p}},A_{\mathfrak{p}}N_{\mathfrak{p}}$
are simply connected closed subgroups of $G$ and $G$ is diffeomorphic
to $K\times A_{\mathfrak{p}}\times N_{\mathfrak{p}}$. \end{itemize}
\begin{rem*}
The inner product $\langle\,,\,\rangle$ restricted to $\mathfrak{a}_{\mathfrak{p}}$
induces an inner product on the dual space $\mathfrak{a}_{\mathfrak{p}}^{*}$
and hence also on the restricted roots.\end{rem*}
\begin{defn*}
1) A positive restricted root $\alpha\in\Sigma$ that cannot be decomposed
as a sum of two positive restricted roots is called a $\emph{simple restricted root}$.
The set of simple restricted roots is denoted by $\triangle$. 

2) An element $\lambda\in\mathfrak{a}_{\mathfrak{p}}^{*}$ is called
a $\emph{dominant weight}$ (resp. $\emph{strongly dominant}$) if
$\langle\lambda,\alpha\rangle\geq0$ (resp. $\langle\lambda,\alpha\rangle>0$)
for any $\alpha\in\triangle$. 

3) Let $G$ be a complex reductive group, $T$ be a maximal torus.
We call a homomorphism $\chi:T\longrightarrow\mathbb{C}$ a $\emph{dominant character}$
(resp. $\emph{strongly dominant}$) if $\mathrm{Re}\langle d\chi|_{\mathfrak{a}_{\mathfrak{p}}},\alpha\rangle\geq0$
(resp. $\mathrm{Re}\langle d\chi|_{\mathfrak{a}_{\mathfrak{p}}},\alpha\rangle>0$)
for any $\alpha\in\triangle$. \end{defn*}
\begin{example*}
Consider $G=\mathrm{GL}_{n}(\mathbb{C})$. Fix $T$ to be the subgroup
of diagonal matrices and a Borel $B$ to be the subgroup of upper
triangular matrices. Choose a compact from by the involution $A\longmapsto\left(A^{*}\right)^{-1}$
where $*$ is conjugate transpose. Under this identification we can
choose $A_{\mathfrak{p}}$ to be the subgroup of real and positive
diagonal matrices and $A$ to be the subgroup of real diagonal matrices.
$\langle X,Y\rangle:=\mathrm{Re}\mathrm{Tr}(X\cdot Y^{*})$ defines
a real valued inner product on $\mathfrak{gl}_{n}$\textcolor{red}{.}
Denote $E_{ij}\in\mathfrak{gl}_{n}$ to be the matrix with value $1$
at the $(i,j)$th entry and $0$ at all other entries. Any character
$\chi:T\longrightarrow\mathbb{C}$ can be written in the form
\[
\chi(t)=\chi(t_{1},...,t_{n})=\prod\left|t_{i}\right|^{\lambda_{i}}\cdot\left(\frac{t_{i}}{\left|t_{i}\right|}\right)^{n_{i}},
\]
for some $\lambda_{i}\in\mathbb{C}$, $n_{i}\in\mathbb{Z}$. We can
write $\mathfrak{gl}_{n}=\mathfrak{g}_{0}\oplus\bigoplus_{\lambda\in\Sigma}\mathfrak{g}_{\alpha_{ij}}$
where $\mathfrak{g}_{\alpha_{ij}}=\mathrm{E}_{ij}$ and $\alpha_{ij}=t_{i}/t_{j}$.
Notice that the positive restricted roots are $\{\alpha_{ij}\}_{i<j}$.
On the Lie algebra level, $\alpha_{ij}(T_{1},...,T_{n})=T_{i}-T_{j}$
where $T_{1},...,T_{n}$ is the corresponding basis for $\mathfrak{a}_{\mathfrak{p}}$.
Note that $\chi|_{A_{\mathfrak{p}}}(t)=\chi(t_{1},...,t_{n})=\prod t_{i}^{\lambda_{i}}$and
hence $d\chi|_{\mathfrak{a}_{\mathfrak{p}}}(T_{1},...,T_{n})=(\lambda_{1}T_{1},...,\lambda_{n}T_{n})$. 

Since $\mathfrak{a}_{\mathfrak{p}}$ consists of diagonal and real
matrices we have that $\mathrm{Tr}(X\cdot Y^{*})=\sum x_{i}\cdot y_{i}$.
The inner product on $\mathfrak{a}_{\mathfrak{p}}^{*}$ is defined
by $\langle\alpha,\beta\rangle:=\langle X_{\alpha},X_{\beta}\rangle$
via the identification $\alpha\longrightarrow\langle X_{\alpha},\,\rangle$
for $X_{\alpha}\in\mathfrak{a}_{\mathfrak{p}}$. Under this identification,
$\alpha_{ij}$ corresponds to the vector $v_{ij}=(0,...,0,1,0,...,0,-1,...0)$
($1$ and $-1$ in the $i$th and $j$th coordinates and 0 in the
others) and $d\chi|_{\mathfrak{a}_{\mathfrak{p}}}$ corresponds to
$(\lambda_{1},...,\lambda_{n})$. Requiring $\chi$ to be dominant
is the same as requiring that for any $i<j$ we have:
\[
0\leq\mathrm{Re}\langle v_{ij},(\lambda_{1},...,\lambda_{n})\rangle=\mathrm{Re}(\lambda_{i})-\mathrm{Re}(\lambda_{j})\geq0,
\]
that is, $\mathrm{Re}(\lambda_{1})\geq...\geq\mathrm{Re}(\lambda_{n})$.
Strong dominance implies that $\mathrm{Re}(\lambda_{1})>...>\mathrm{Re}(\lambda_{n})$.
\end{example*}

\subsection{\label{subC1:Parabolic-subgroups}Parabolic subgroups}

We now introduce some facts about parabolic subgroups:
\begin{defn*}
A $\emph{parabolic subgroup}$ of an algebraic group $\underline{G}$
defined over $k$ is a closed subgroup $\underline{P}\subseteq\underline{G}$,
for which the quotient space $\underline{G}/\underline{P}$ is a projective
algebraic variety. A subgroup $P$ of $G=\underline{G}(k)$ is parabolic
if $P=\underline{P}(k)$ for some parabolic subgroup $\underline{P}\subseteq\underline{G}$,
such that $\underline{P}(k)$ is dense in $\underline{P}$ (in the
Zarisky topology).

There always exist minimal parabolic subgroups and they have the following
structure theory over $\mathbb{R}$: let $G$ be a linear connected
reductive group, $K$ a maximal compact subgroup, $A_{\mathfrak{p}}$,
$N_{\mathfrak{p}}$ as defined above and $\mathfrak{g}=\mathfrak{k}\oplus\mathfrak{p}=\mathfrak{g}_{0}\oplus\bigoplus_{\lambda\in\Sigma}\mathfrak{g}_{\lambda}$.
Consider $M_{\mathfrak{p}}=C_{K}(\mathfrak{a}_{\mathfrak{p}})$, i.e
the set of $k\in K$ such that $\mathrm{Ad}_{k}=\mathrm{id}$ on $\mathfrak{a}_{\mathfrak{p}}$.
We have the following properties:\end{defn*}
\begin{enumerate}
\item $M_{\mathfrak{p}}$ is a closed subgroup of $K$, hence compact. 
\item $M_{\mathfrak{p}}$ centralizes $\mathfrak{a}_{\mathfrak{p}}$ and
normalizes each $\mathfrak{g}_{\lambda}$.
\item $M_{\mathfrak{p}}$ centralizes $A$ and normalizes $N_{\mathfrak{p}}$. 
\item $M_{\mathfrak{p}}A_{\mathfrak{p}}N_{\mathfrak{p}}$ is a closed minimal
parabolic subgroup of $G$. 
\item Every minimal parabolic subgroup of $G$ can be represented in this
form (4). 
\end{enumerate}
Fix a minimal parabolic $Q_{\mathfrak{p}}:=M_{\mathfrak{p}}A_{\mathfrak{p}}N_{\mathfrak{p}}$.
A $\emph{Standard}$ parabolic subgroup is any closed subgroup $Q$
that contains $Q_{\mathfrak{p}}$. Any standard parabolic subgroup
has a $\emph{Langlands decomposition}$ of the form $Q=MAN$ obtained
as follows: 
\begin{itemize}
\item $MA:=Q\cap\Theta(Q)$, where $\Theta$ is the global Cartan involution
that fixes $K$.
\item $A:=Z(MA)$ with Lie algebra $\mathfrak{a}$.
\item $\mathfrak{m}$ is the orthogonal complement of $\mathfrak{a}$ in
$\mathfrak{m}\oplus\mathfrak{a}$ with respect to $\langle,\rangle$
defined above. 
\item $M_{0}$ is the analytic subgroup that corresponds to $\mathfrak{m}$.
$M=C_{K}(A)\cdot M_{0}$ (it is non compact if $Q\neq Q_{\mathfrak{p}}$).
\item Let $\mathfrak{n}$ be the direct sum of eigenspaces in $\mathfrak{q}=Lie(Q)$
with nonzero eigenvalues corresponds to $\mathfrak{a}$, and define
$N$ to be the analytic subgroup corresponding to $\mathfrak{n}$. 
\end{itemize}
This describes the structure of any standard parabolic. There is another
description of standard parabolic subgroups, from which one can understand
the connection between $\mathfrak{a}_{\mathfrak{p}},\mathfrak{m}_{\mathfrak{p}},\mathfrak{n}_{\mathfrak{p}}$
and $\mathfrak{a},\mathfrak{m},\mathfrak{n}$: let $F$ be any subset
of $\triangle$. Denote $\Sigma_{F}=\{\beta\in\Sigma|\beta\in Span(F)\}$,
$\Gamma_{F}=\Sigma^{+}\cup\Sigma_{F}$ and define: 
\[
\mathfrak{q}_{F}=\mathfrak{m}_{\mathfrak{p}}\oplus\mathfrak{a}_{\mathfrak{p}}\oplus\bigoplus_{\beta\in\Gamma_{F}}\mathfrak{g}_{\beta}.
\]
This is a parabolic subalgebra of $\mathfrak{g}$ containing $\mathfrak{q}_{\mathfrak{p}}=\mathfrak{m}_{\mathfrak{p}}\oplus\mathfrak{a}_{\mathfrak{p}}\oplus\mathfrak{n}_{\mathfrak{p}}$.
Define:
\begin{enumerate}
\item $\mathfrak{a}_{F}:=\cap_{\beta\in\Sigma_{F}}\mathrm{ker}\beta$. It
is contained in $\mathfrak{a}_{\mathfrak{p}}$. 
\item $\mathfrak{a}_{M,F}:=\mathfrak{a}_{F}^{\bot}\subseteq\mathfrak{a}_{\mathfrak{p}}$. 
\item $\mathfrak{m}_{F}:=\mathfrak{a}_{M,F}\oplus\mathfrak{m}_{\mathfrak{p}}\oplus\bigoplus_{\beta\in\Sigma_{F}}\mathfrak{g}_{\beta}$.
\item $\mathfrak{n}_{F}:=\bigoplus_{\beta\in\Sigma^{+}\smallsetminus\Sigma_{F}}\mathfrak{g}_{\beta}$.
\item $\mathfrak{n}_{M,F}:=\mathfrak{n}_{p}\cap\mathfrak{m}_{F}$.
\end{enumerate}
and we get that $\mathfrak{q}_{F}=\mathfrak{m}_{F}\oplus\mathfrak{a}_{F}\oplus\mathfrak{n}_{F}$.
This decomposition has some nice properties: 
\begin{itemize}
\item The centralizer of $\mathfrak{a}_{F}$ in $\mathfrak{g}$ is $\mathfrak{m}_{F}\oplus\mathfrak{a}_{F}$. 
\item $\mathfrak{a}_{\mathfrak{p}}=\mathfrak{a}_{F}\oplus\mathfrak{a}_{M,F}$.
\item $\mathfrak{n}_{\mathfrak{p}}=\mathfrak{n}_{F}\oplus\mathfrak{n}_{M,F}$.
\end{itemize}
Now let $A_{F}$ and $N_{F}$ be the analytic subgroups with Lie algebras
$\mathfrak{a}_{F},\mathfrak{n}_{F}$. Define $M_{0,F}$ to be the
analytic subgroup that correspond to $\mathfrak{m}_{F}$ and $M=C_{K}(A_{F})\cdot M_{0}$.
We get a standard parabolic subgroup $Q_{F}:=M_{F}A_{F}N_{F}$ with
Lie algebra $\mathfrak{q}_{F}=\mathfrak{m}_{F}\oplus\mathfrak{a}_{F}\oplus\mathfrak{n}_{F}$.
This describes all standard parabolic subgroups. Notice that if $G$
is semisimple then $\mathrm{dim}A_{\mathfrak{p}}=\left|\triangle\right|$
and there are exactly $2^{dimA_{\mathfrak{p}}}$ different subsets
$F\subseteq\triangle$ and the number of standard parabolic subgroups
is $2^{\mathrm{dim}A_{p}}$. 
\begin{example*}
Consider $G=\mathrm{SL_{n}}(\mathbb{R})$ and fix a minimal parabolic
subgroup $M_{\mathfrak{p}}A_{\mathfrak{p}}N_{\mathfrak{p}}$ corresponding
to the upper triangular matrices. The standard parabolic subgroups
$Q_{F}$ are block upper triangular matrices, where different choices
of $F$ corresponds to different matrix block partitions. Here, $M_{F}A_{F}$
corresponds to block diagonal matrices. When $F$ is empty, we get
that $Q_{F}=Q_{\mathfrak{p}}$, the minimal parabolic subgroup and
when $F=\triangle$ we get that $Q_{F}=G$, $A,N=\{e\}$ and $M=G$. 
\end{example*}

\section{\label{Append:Restriction-of-scalars}Restriction of scalars and
Galois involution}

Let $k$ be a field, $\underline{G}$ be an algebraic group defined
over $k$. It can be written as $\underline{G}=\mathrm{Spec}(A)$
where $A=k[x_{1},...,x_{n}]/\langle f_{1},...,f_{m}\rangle$ is a
$k$-algebra. Let $K\supset k$ be a field extension and define $\underline{G}(K):=\mathrm{Hom}_{k}(A,K)$
. This set is called the $\emph{K points of \ensuremath{\underline{G}}}$
and it is in bijection with the solutions of $\{f_{1}(x_{1},...,x_{n})=0,...,f_{m}(x_{1},...,x_{n})=0\}$
in $K$. $\underline{G}(K)$ has a structure of a group induced by
the group structure of $\underline{G}$. Notice that the $\mathbb{C}$
(resp. $\mathbb{R}$)-points of an algebraic group $\underline{G}$
defined over $\mathbb{R}$ is a complex (resp. real) Lie group of
the same dimension. 

Now set $k=\mathbb{R}$. A real structure on $\underline{G}$ is determined
by an action of $\mathrm{Gal}(\mathbb{C}/\mathbb{R})=\{\mathrm{id},\tau\}$
on its coordinate ring $\mathbb{C}[\underline{G}]$ , such that $\mathbb{C}[\underline{G}]^{\tau}\otimes_{\mathbb{R}}\mathbb{C}\simeq\mathbb{C}[\underline{G}]$.
We denote $\mathbb{R}[\underline{G}]:=\mathbb{C}[\underline{G}]^{\tau}$
as the real structure. $\mathrm{Gal}(\mathbb{C}/\mathbb{R})$ also
acts on $\underline{G}(\mathbb{C})$ such that $\left(\underline{G}(\mathbb{C})\right)^{\sigma}$
is isomorphic to $\mathrm{Hom}_{\mathbb{R}}(\mathbb{R}[\underline{G}],\mathbb{R})$.
A motivating example is $\underline{G}=\mathrm{GL}_{n}$, $\mathbb{C}[\underline{G}]=\mathbb{C}[t_{11},...,t_{nn},det^{-1}]$,
$\mathbb{R}[\underline{G}]=\mathbb{R}[t_{11},...,t_{nn},det^{-1}]=\mathbb{C}[\underline{G}]^{\tau}$,
where $\tau$ acts by complex conjugation. 

Note that $\mathrm{Gal}(\mathbb{C}/\mathbb{R})$ acts by an involution
$\tau$ on $\mathbb{C}[\underline{G}]$, but it is not a $\mathbb{C}$-algebra
automorphism as it is not $\mathbb{C}$-linear. In order to view $\tau$
as a morphism of algebraic group, we need to introduce restriction
of scalars. 
\begin{fact*}
There exists a functor $\mathrm{Res}_{\mathbb{C}/\mathbb{R}}$ (called
$\emph{restriction of scalars}$) from the category of affine group
schemes over $\mathbb{C}$ to the category of affine group schemes
over $\mathbb{R}$, that is right adjoint to base change, i.e: 
\[
\mathrm{Hom}(G,\mathrm{Res}_{\mathbb{C}/\mathbb{R}}(G'))\simeq\mathrm{Hom}(G\times_{\mathrm{Spec}\mathbb{R}}\mathrm{Spec}\mathbb{C},G'),
\]
 where $G$ and $G'$ are algebraic groups defined over $\mathbb{R}$
and $\mathbb{C}$, respectively.
\end{fact*}
We can now apply the following procedure. Let $\underline{G}$ be
an algebraic group defined over $\mathbb{R}$, with a real structure
$\mathbb{C}[\underline{G}]^{\tau}$ for $\tau\in\mathrm{Gal}(\mathbb{C}/\mathbb{R})$.
Consider $\underline{G}$ as an algebraic group over $\mathbb{C}$
(denoted $\underline{G}_{\mathbb{C}}$) and apply the functor $\mathrm{Res}_{\mathbb{C}/\mathbb{R}}$
to get an algebraic group $\underline{G}_{\mathbb{C}/\mathbb{R}}:=\mathrm{Res}_{\mathbb{C}/\mathbb{R}}(\underline{G}_{\mathbb{C}})$
defined over $\mathbb{R}$. We have the following properties: 

{*} $\underline{G}(\mathbb{C})=\underline{G}_{\mathbb{C}/\mathbb{R}}(\mathbb{R})$.

{*} The map $\tau:\underline{G}\longrightarrow\underline{G}$ induces
an involution of $\mathbb{R}$-algebraic groups $\theta:\underline{G}_{\mathbb{C}/\mathbb{R}}\longrightarrow\underline{G}_{\mathbb{C}/\mathbb{R}}$
such that $\underline{G}_{\mathbb{C}/\mathbb{R}}^{\theta}(\mathbb{R})=\underline{G}(\mathbb{R})$.
\begin{defn*}
We call the involution $\theta:\underline{G}_{\mathbb{C}/\mathbb{R}}\longrightarrow\underline{G}_{\mathbb{C}/\mathbb{R}}$
defined above a $\emph{Galois involution}$. Note that $\underline{G}_{\mathbb{C}/\mathbb{R}}(\mathbb{R})/\underline{G}_{\mathbb{C}/\mathbb{R}}^{\theta}(\mathbb{R})\simeq\underline{G}(\mathbb{C})/\underline{G}(\mathbb{R})$.\end{defn*}
\begin{thm}
\label{thmC.1:Any-Galois-involution}Let $\underline{G}$ be an algebraic
group defined over $\mathbb{R}$, $\underline{G}_{\mathbb{C}/\mathbb{R}}$
be the restriction of scalars of its complexification, and $\theta:\underline{G}_{\mathbb{C}/\mathbb{R}}\longrightarrow\underline{G}_{\mathbb{C}/\mathbb{R}}$
be a Galois involution. Let $G=\underline{G}_{\mathbb{C}/\mathbb{R}}(\mathbb{R})$
and $H=\underline{G}_{\mathbb{C}/\mathbb{R}}^{\theta}(\mathbb{R})$.
Then $\theta:G\longrightarrow G$ has the property $(\star)$ (see
Section \ref{sub:Main-results-and}).\end{thm}
\begin{proof}
Fix a Borel subgroup $B$, a $\theta$-stable maximal torus $T$ and
a $\theta$-stable maximal $\mathbb{R}$-split torus $A$, all in
$G$, such that $A_{\mathfrak{p}}\subseteq A\subseteq T\subseteq B$
(see Proposition \ref{prop6.1:We-can-choose}). If we consider $G$
as a complex group then we can write $\mathrm{Lie}(G)=\mathfrak{g}=\mathfrak{g}_{0}\oplus\bigoplus_{\widetilde{\alpha}\in\Sigma(G,T)}\mathfrak{g}_{\widetilde{\alpha}}$,
where $\mathfrak{g}_{0}=\mathrm{Lie}(T)$, $\Sigma(G,T)$ is the corresponding
root system and each eigenspace $\mathfrak{g}_{\widetilde{\alpha}}$
is one dimensional over $\mathbb{C}$. If we consider $G$ as a real
group then we can write $\mathfrak{g}=\mathfrak{g}_{0}\oplus\bigoplus_{\alpha\in\Sigma(G,A_{\mathfrak{p}})}\mathfrak{g}_{\alpha}$.
$\mathfrak{g}_{\alpha}$ becomes a two dimensional space over $\mathbb{R}$
and $\widetilde{\alpha}|_{A_{\mathfrak{p}}}=\alpha$. Notice that
$\mathrm{dim}H=\frac{1}{2}\cdot\mathrm{dim}G$ and $\mathfrak{g}^{\theta}=\mathfrak{g}_{0}^{\theta}\oplus\bigoplus_{\alpha\in\Sigma(G,T)}\mathfrak{g}_{\alpha}$
has exactly half of the dimension in each root space. 

Note that for any $n\in N_{G}(T)$, $\mathrm{Ad}_{n}$ is a $\mathbb{C}$-linear
map, and $d\theta$ is a semi-linear map. This implies that if $\theta_{n}(\alpha)=\alpha$
then $d(\theta_{n})|_{\mathfrak{g}_{\alpha}}=\mathrm{Ad}_{n}\circ d\theta|_{\mathfrak{g}_{\alpha}}\neq\pm\mathrm{Id}$. 

We now show that $\delta_{B^{\theta_{n}}}=\delta_{B}^{1/2}|_{B^{\theta_{n}}}$
as a character of $B^{\theta_{n}}$. Notice that $\delta_{B}(b)=\left|\mathrm{det}(\mathrm{Ad}_{b})\right|$,
where $b\in B$, $\mathrm{Ad}_{b}\in\mathrm{GL}(\mathfrak{b})$ and
$\mathfrak{b}:=\mathrm{Lie}(B)$. As modular characters are trivial
on the unipotent part of $B^{\theta_{n}}$, it is enough to prove
that $\delta_{B^{\theta_{n}}}|_{T^{\theta_{n}}}=\delta_{B}^{1/2}|_{T^{\theta_{n}}}$.
Note that $S:=\{n\theta(t)n^{-1}t|t\in T\}$ is a connected component
of $T^{\theta_{n}}$, so it is enough to prove that $\delta_{B^{\theta_{n}}}|_{S}=\delta_{B}^{1/2}|_{S}$.
As $T$ is $\mathbb{C}$-split, and by the fact that modular characters
are trivial on compact subgroups, it is enough to prove that $\delta_{B^{\theta_{n}}}|_{\widetilde{S}}=\delta_{B}^{1/2}|_{\widetilde{S}}$
where $\widetilde{S}=\{n\theta(a)n^{-1}a|a\in A_{\mathfrak{p}}\}$.
Denote
\[
S_{1}:=\{\alpha\in\Sigma^{+}(G,A_{\mathfrak{p}})\text{ s.t }\alpha\neq\theta_{n}(\alpha)\in\Sigma^{+}(G,A_{\mathfrak{p}})\},
\]
\[
S_{2}:=\{\alpha\in\Sigma^{+}(G,A_{\mathfrak{p}})\text{ s.t }\alpha=\theta_{n}(\alpha)\in\Sigma^{+}(G,A_{\mathfrak{p}})\},
\]
and
\[
S_{3}:=\{\alpha\in\Sigma^{+}(G,A_{\mathfrak{p}})\text{ s.t }\theta_{n}(\alpha)\in\Sigma^{-}(G,A_{\mathfrak{p}})\}.
\]

Note that for any $a\in A_{\mathfrak{p}}$, 
\[
\delta_{B}(a)=\left|\mathrm{det}(\mathrm{Ad}_{a})\right|=\left|\prod_{\alpha\in\Sigma^{+}(G,A_{\mathfrak{p}})}\left(\alpha(a)\right)^{2}\right|,
\]
and therefore 
\[
\delta_{B}^{1/2}(a\theta_{n}(a))=\left|\prod_{\alpha\in\Sigma^{+}(G,A_{\mathfrak{p}})}\alpha(a)\cdot\theta_{n}(\alpha)(a)\right|.
\]
$\mathrm{Lie}(B^{\theta_{n}})$ consists of the $+1$ eigenspace of
$\mathfrak{b}$ under $d(\theta_{n})$, hence it is clear that roots
in $S_{3}$ doesn't contribute to $\mathrm{Lie}(B^{\theta_{n}})$,
that is, $\mathrm{Lie}(B^{\theta_{n}})=\left(\mathfrak{g}_{0}\oplus\bigoplus_{\alpha\in S_{1},S_{2}}\mathfrak{g}_{\alpha}\right)_{d(\theta_{n}),+1}$,
where $\left(\,\right)_{d(\theta_{n}),+1}$ denote the $+1$ eigenspace
of $d(\theta_{n})$. As $d(\theta_{n})|_{\mathfrak{g}_{\alpha}}\neq\pm\mathrm{Id}$
for roots $\alpha\in S_{2}$ we deduce that: 
\[
\delta_{B^{\theta_{n}}}(a\theta_{n}(a))=\left|\prod_{\alpha\in S_{1},S_{2}}\left(\alpha(a)\cdot\theta_{n}(\alpha)(a)\right)\right|.
\]

Hence, it is enough to prove that $\prod_{\alpha\in S_{3}}\alpha(a)\cdot\theta_{n}(\alpha)(a)$
is trivial for any $a\in A_{\mathfrak{p}}$, or equivalently by taking
differentials, $\sum_{\alpha\in S_{3}}\alpha(X)+\theta_{n}(\alpha)(X)=0$
for any $X\in\mathfrak{a}_{\mathfrak{p}}=\mathrm{Lie}(A_{\mathfrak{p}})$.
Indeed, if $\alpha\in S_{3}$ then also $-\theta_{n}(\alpha)\in S_{3}$
as well, hence we are done. \end{proof}

\end{document}